\documentclass[11pt,letterpaper]{amsart}
\usepackage{amsmath}
\usepackage{amssymb}
\usepackage{amsthm}
\usepackage[msc-links,abbrev]{amsrefs}
\usepackage{hyperref}
\usepackage[noabbrev,capitalize]{cleveref}
\usepackage{mathrsfs}
\usepackage{mathtools}
\usepackage{slashed}
\usepackage{tikz-cd}
\usepackage{enumitem}

\setenumerate{label=(\roman*)}

\DeclareMathOperator{\Id}{Id}

\DeclareMathOperator{\im}{im}
\DeclareMathOperator{\tr}{tr}
\DeclareMathOperator{\Vol}{Vol}
\DeclareMathOperator{\dvol}{dV}

\DeclareMathOperator{\Ric}{Ric}

\newcommand{\Sch}{\mathsf{P}}

\newcommand{\cg}{\widetilde{g}}

\newcommand{\cH}{\widetilde{H}}

\newcommand{\cL}{\widetilde{L}}

\newcommand{\cN}{\widetilde{N}}

\newcommand{\mB}{\mathcal{B}}

\newcommand{\mE}{\mathcal{E}}

\newcommand{\mM}{\mathcal{M}}

\newcommand{\mP}{\mathcal{P}}
\newcommand{\mQ}{\mathcal{Q}}

\newcommand{\mU}{\mathcal{U}}
\newcommand{\mV}{\mathcal{V}}
\newcommand{\mW}{\mathcal{W}}

\newcommand{\kg}{\mathfrak{g}}

\newcommand{\bB}{\mathbb{B}}

\newcommand{\bN}{\mathbb{N}}

\newcommand{\bR}{\mathbb{R}}
\newcommand{\bS}{\mathbb{S}}

\newcommand{\bZ}{\mathbb{Z}}

\def\sideremark#1{\ifvmode\leavevmode\fi\vadjust{\vbox to0pt{\vss
 \hbox to 0pt{\hskip\hsize\hskip1em
 \vbox{\hsize3cm\tiny\raggedright\pretolerance10000
 \noindent #1\hfill}\hss}\vbox to8pt{\vfil}\vss}}}

\newtheorem{theorem}{Theorem}[section]

\newtheorem{lemma}[theorem]{Lemma}

\theoremstyle{definition}

\newtheorem{question}[theorem]{Question}

\theoremstyle{remark}
\newtheorem{remark}[theorem]{Remark}

\numberwithin{equation}{section}

\makeatletter
\@namedef{subjclassname@2020}{\textup{2020} Mathematics Subject Classification}
\makeatother

\begin{document}

\title{Energy reduction for Fourth order Willmore energy}

\author[N. Wu]{Nan Wu}
\address[N. Wu]{Department of Mathematics \\ University of Notre Dame \\ Notre Dame \\ IN 46656 \\ USA}
\email{nwu2@nd.edu}

\author[Z. Yan]{Zetian Yan}
\address[Z. Yan]{Department of Mathematics \\ UC Santa Barbara \\ Santa Barbara \\ CA 93106 \\ USA}
\email{ztyan@ucsb.edu}

\begin{abstract}
We introduce a fourth-order Willmore-type problem for closed four-dimensional submanifolds immersed in $\bR^n$ and establish a connected sum energy reduction for the general fourth-order Willmore energy, analogous to the seminal result of Bauer and Kuwert \cite{Bauer-Kuwert03}. 
 
\end{abstract}
\maketitle

\section{introduction}
The Willmore energy, named after the mathematician Thomas Willmore, is a quantitative measure of how much an immersed surface deviates from a round sphere.
For an immersed closed surface \( f \colon \Sigma^2 \hookrightarrow \mathbb{R}^n \), the Willmore energy is defined by
\begin{equation*}
    \mathcal{W}(f) := \int_{\Sigma} |H|^2 \, \dvol_g,
\end{equation*}
where \( H \) and \( g \) denote the averaged mean curvature vector and the induced metric of \( f \), respectively.
Willmore showed that round
spheres have the least possible Willmore energy among all compact surfaces
in $\mathbb{R}^3$. More precisely, for every immersed compact surface $f:\Sigma \to \mathbb{R}^n$, the Willmore energy satisfies
\begin{equation*}
W(f) \geq 4 \pi
\end{equation*}
with equality only for round spheres. If \( \mathcal{W}(f) < 8\pi \), then \( f \) is an embedding by a result of Li and Yau \cite{Li-Yau82}. Furthermore, Bryant \cite{Bryant84} classified all Willmore spheres in codimension one.

An interesting question is to investigate the minimizing problem for the Willmore
energy over all closed surfaces in $\mathbb{R}^n$ of a given topological type.
To be precise, for fixed genus \( p \geq 0 \) and ambient dimension \( n \geq 3 \), we define
\begin{equation*}
    \beta^n_p := \inf_{\mathcal{S}^n_p} \mathcal{W}(f),
\end{equation*}
where \( \mathcal{S}^n_p \) denotes the set of all immersions from closed orientable surfaces of genus \( p \) in $\mathbb{R}^n$. 
It was proved independently by Kusner and Pinkall (see \cite{kuhnel1986total} and \cite{Kusner}) that a Lawson minimal surface of genus $p$, for every $p$, has area strictly smaller than $8\pi$. This implies $\beta_{p}^n < 8\pi$. Later Kuwert, Li and Schatzle \cite{Kuwert2010TheLG} showed that $\beta^n_p$ tends to $8\pi$ as $p$ tends
to infinity.

In \cite{Simon93}, Simon proved the existence of a smooth minimizing Willmore torus attaining $\beta^n_1$ by geometric measure theory and proposed a criterion to verify the existence of such minimizing  Willmore surfaces with fixed topology of genus $p\geq 2$:
\begin{theorem}[Simon's Criterion]
    For any \( p \in \mathbb{N} \), the infimum \( \beta^n_p \) is achieved by a closed embedded surface if
    \begin{equation}\label{eq:topological-criterion}
        e^n_p < \sum_{r} e^n_{p_r}, \quad \text{where } e^n_p := \beta^n_p - 4\pi,
    \end{equation}
    holds for any partition \( p = p_1 + \cdots + p_r \), and \( \beta^n_{p_r} \) corresponds to a Willmore surface of genus \( p_r \).
\end{theorem}

Later, Bauer and Kuwert\cite[Theorem 1.3]{Bauer-Kuwert03} verified the validity of \eqref{eq:topological-criterion} by constructing test surfaces via connected sums:
\begin{theorem}[Connected Sum Energy Reduction]
    Let \( f_i \colon \Sigma_i^2 \hookrightarrow \mathbb{R}^n \), \( i = 1, 2 \), be two smoothly immersed closed surfaces. If neither \( f_1 \) nor \( f_2 \) is a round sphere (i.e., totally umbilic), there exists an immersed surface \( f \colon \Sigma^2 \hookrightarrow \mathbb{R}^n \) with topological type of \( \Sigma^2 = \Sigma_1^2 \# \Sigma_2^2 \) such that
    \begin{equation}\label{Connected-Sum-Energy-Reduction-2}
        \mathcal{W}(f) < \mathcal{W}(f_1) + \mathcal{W}(f_2) - 4\pi.
    \end{equation}
\end{theorem}

For further studies on the regularity of Willmore surfaces, removable singularities, and related topics, we refer readers to \cite{Bernard-Riviere14, Kuwert-SchatzleS04, Riviere08, Laurain2016EnergyQO} and references therein.

In the past decade, conformal geometers and physicists were searching for extrinsic conformal invariants on four-dimensional immersed submanifolds. It was first discovered by Guven \cite{Guven05} on hypersurfaces in $\bR^5$, and extended to higher codimensions and curved ambient spaces by Graham--Reichert \cite{Graham-Reichert20} and Zhang \cite{Zhang21} via the renormalized area in Poincar\'{e}--Einstein metric. Main results in aforementioned articles imply that for a closed immersion $\Phi: \Sigma^4 \hookrightarrow \bR^n$, the functional
\begin{equation}
    \mE_{GR}(\Phi):=\int_{\Sigma} \left(\big|\nabla^{\perp}H\big|^2-\big|L^{t}H\big|^2+7|H|^4\right)\dvol,
\end{equation}
where $L$ is the second fundamental form and $\nabla^{\perp}$ is the normal connection, is invariant under conformal transformations on $\bR^n$. $\mE_{GR}$ is a multiple of the fourth order Willmore energy in \cite{Graham-Reichert20, Zhang21}. In particular, Graham and Reichert demonstrated that $\mE_{GR}$ is neither bounded from below nor above for immersions of manifolds with certain topological types-such as 
$\mathbb{S}^1 \times \mathbb{S}^3$ and $\mathbb{S}^2 \times \mathbb{S}^2$-while proving that the round sphere remains stable; see \cite{Graham-Reichert20} for further details. Furthermore, Martino showed that $\mathcal{E}_{GR}$ is unbounded below on any closed hypersurface $\Sigma^4 \to \mathbb{R}^5$ in \cite{Martino23}. After the first publication of Martino's result, Graham–Reichert mentioned that they already had a proof in a personal communication. The physical interpretation for such unboundedness feature was shown in \cite{Anastasiou2025}.

For embeddings of four-manifolds into $\bR^n$ with $n \geq 6$, the space of global conformal invariants is six-dimensional: a basis is $\mE_{GR}$, the invariant $\mathcal{I}$ defined in \cite{CGKTW25} and four contractions of $\mathring{L}$ (the traceless part of the second fundamental form) which are obtained by breaking $\mathring{L}_{\alpha\beta}^{\alpha'}\mathring{L}_{\gamma\delta}^{\beta'}$ into symmetric and antisymmetric-in-normal pieces and then taking possible contractions; see \cite[Section 6]{CGKTW25} for more details. While for embeddings into $\bR^5$, the space of global conformal invariants is four-dimensional and spanned by 
\begin{equation*}
 \mE_{GR},\quad   \int_{\Sigma} |\mathring{L}|^4\dvol,\quad  \int_{\Sigma}|\mathring{L}^2|^2 \dvol, \quad \chi(\Sigma^4);
\end{equation*}
see \cite{Mondino-Nguyen18} for precise statement.

In order to fix aforementioned unbounded issue for $\mE_{GR}$, the modified invariant considered now is in the form of 
\begin{equation}\label{general-fourth-order-Willmore-energy}
   \mE^{(\mu,\nu)}(\Phi):=\mE_{GR}(\Phi)+\mu \int_{\Sigma} |\mathring{L}|^4\dvol+\nu \int_{\Sigma}|\mathring{L}^2|^2 \dvol,
\end{equation}
for $\mu, \nu \in \bR.$ In particular, by setting $\nu=0$ and $\mu \geq \frac{1}{12}$, we have $ \mE^{(\mu,0)}\geq 8\pi^2$ with equality if and only if $\Phi$ parametrizes a round sphere.   We refer readers to \cite{Blitz-Gover-Waldon24,Gover-Waldon20,Case-Graham-Kuo} for further studies on algebraic properties, physical interpretations, etc. of $\mE^{(\mu,\nu)}$. 

In this paper, we derive the following connected sum energy reduction inequality for general fourth order Willmore energy $\mE^{(\mu,\nu)}$. 
\begin{theorem}\label{Main 1}
    Let $\Phi_i: \Sigma^4_i \to \mathbb{R}^n $ be two smoothly immersed, closed submanifolds in $\mathbb{R}^n$. If neither $\Phi_1$ or $\Phi_2$ is not a round sphere, then there is an immersed submanifold $\Phi: \Sigma^4 \to \mathbb{R}^n$ with topological type of the connected sum $\Sigma_1 \# \Sigma_2$, such that for any $\mu, \nu \in \bR$,
    \begin{equation}\label{Connected-Sum-Energy-Reduction-4}
         \mathcal{E}^{(\mu, \nu)}(\Phi) < \mathcal{E}^{(\mu, \nu)}(\Phi_1) + \mathcal{E}^{(\mu, \nu)}(\Phi_2) -8\pi^2.
    \end{equation}
\end{theorem}

For proving Theorem \ref{Main 1}, the key ingredient is the following energy identity under the inversion, which is inspired by Theorem 2.2 in \cite{Bauer-Kuwert03}:
\begin{theorem}\label{Main 2}
    Let $\Phi: \Sigma^4\to \bR^n$ be a closed immersed submanifold with $0\in \Phi(\Sigma)$. If $\hat{\Phi}$ is given by inverting $\Phi$ at the sphere of radius $1$ around $0$, i.e., $\hat{\Phi}:=I\circ \Phi$, then 
    \begin{equation}\label{loss of energy}
        \begin{split}
    \mE^{(\mu, \nu)}(\hat{\Phi})&=\mE^{(\mu, \nu)}(\Phi)-8\pi^2\cdot \mathrm{card} \Phi^{-1}(0)\\
    &=\mE^{(\mu, \nu)}(\Phi)-\mE^{(\mu, \nu)}(\mathbb{S}^4)\cdot \mathrm{card} \Phi^{-1}(0).
    \end{split}
    \end{equation}
\end{theorem}

Several comments are in order. First, to compute residues of singularities generated by the inversion, we reformulate $\mE_{GR}$ in terms of the total extrinsic $Q_4$ curvature and perform the transformation law of the extrinsic Paneitz operator $P_4$; see \cite{Case-Graham-Kuo} for more details. Next, the geometric interpretation of (\ref{loss of energy}) is a four-dimensional Huber's theorem and the change of $\chi(\Sigma^4)$ under inversions, which was implicitly demonstrated in dimension two in \cite{Kuwert-SchatzleS04}. Besides, it suggests that the leading term $\int |\nabla H|^2$ in $\mE^{(\mu,\nu)}$ has no contribution to the loss of energy in (\ref{loss of energy}), which is different from the classical case. 

%Authors with our collaborator are working on the $\epsilon-$regularity of fourth order Willmore submanifolds (minimizers of $\gamma^n_{\chi}$) in the forthcoming paper. In particular, the energy reduction in Theorem \ref{Main 1} will be used to investigate topology of fourth order Willmore submanifolds in our further study.

Inspired by the existence of Willmore surfaces (minimizers) with prescribed genus \cite{Bauer-Kuwert03,Simon93} and our main results, researchers have sought higher-order analogues of the Willmore energy which can be used to define a distinguished class of conformally embedded four-manifolds. More precisely, one may ask the following question:
\begin{question}
For fixed immersed closed manifold $\Sigma^4$, we define  
\begin{equation*}
    \gamma^{(\mu, \nu)}_{\chi(\Sigma)} := \inf_{\Phi \in \mathcal{S}^n_{\chi(\Sigma)}} \mE^{(\mu, \nu)} (\Phi),
\end{equation*}
where $\mathcal{S}^n_{\chi(\Sigma)}$ denotes the set of all immersions from closed four-dimensional submanifolds $M^4$ in $\mathbb{R}^n$ with $\chi(M^4)=\chi(\Sigma^4)$. Is there a pair $(\mu,\nu)$ such that for any fixed four-manifold $\Sigma^4$ which can be immersed in $\mathbb{R}^n$, there exists an immersion $\Phi \colon M^4 \to \mathbb{R}^n$ in $S_{\chi(\Sigma)}^n$ which minimizes $\gamma^{(\mu, \nu)}_{\chi(\Sigma)}$?
\end{question}

In the same spirit of \cite{Bauer-Kuwert03}, the rest proof of Theorem \ref{Main 1} is devoted to the construction of the connected sum. To be more specific, we blow down the inverted immersion $\hat{\Phi}_1$, blow up another immersion $\Phi_2$ and use the triharmonic function with prescribed boundary conditions to do the interpolation in the neck region $[\gamma, 1]\times \bS^3$. Note that for a smooth immersion $\Phi$ satisfying
\begin{equation*}
    \Phi(0)=0,\quad D\Phi(0)=0,
\end{equation*}
the asymptotic decay of each end in $\hat{\Phi}$ is $|x|^{-2}$ as $|x|\to \infty$. Due to this fact, the energy saving of $\hat{\Phi}_{1,\alpha}:=\alpha \Phi_1(\alpha^{-1}\cdot)$ is $O(\alpha^2)$, $\alpha<<1$, while the energy saving of $\Phi_{2,\beta^{-1}}:=\beta^{-1}\Phi_2(\beta\cdot)$ is $O(\beta^4), \beta<<1$ in that the homogeneity of $\mE^{(\mu,\nu)}$ is equal to four. Therefore, the energy reduction (\ref{Connected-Sum-Energy-Reduction-4}) only depends on the energy saving of $\hat{\Phi}_{\alpha}$ and the interaction term between $\hat{\Phi}_{1,\alpha}$, $\Phi_{2,\beta^{-1}}$. An algebraic lemma in Section \ref{Conclusion} yields the desired result.

This paper is organized as follows:

In Section \ref{preliminary} we review necessary backgroud material on the submanifold geometry and extrinsic Paneitz operator. Section 3 is devoted to prove the energy identity (\ref{loss of energy}) and demonstrate its geometric interpretation via the Gauss--Bonnet--Chern formula (\ref{complete-GBC}). Section 4 contains two technical lemmas. We determine the required triharmonic function for the interpolation in Section \ref{connected-sum} and calculate relevant bilinear forms in Section \ref{bilinear}. We compute the energy reduction (\ref{Connected-Sum-Energy-Reduction-4}) in Section \ref{energy} and complete the proof of Theorem \ref{Main 1} in Section \ref{Conclusion}.

\section{Preliminaries}\label{preliminary}
\subsection{Submanifold geometry}
For a Riemannian manifold $(M^n,g)$, we denote the Levi--Civita connection by $\nabla^{g}$, the curvature tensor by $R_{ijkl}$, the Ricci tensor by $\Ric$ or $R_{ij}=R^{k}_{\;\;ikj}$, and the scalar curvature by $R=R^i_{\;i}$. The Schouten tensor of $(M^n,g)$ is
\begin{equation*}
    \Sch_{ij}=\frac{1}{n-2}\left(R_{ij}-\frac{R}{2(n-1)}g_{ij}\right)
\end{equation*}
and the Weyl tensor is defined by the decomposition
\begin{equation*}
    R_{ijkl}=W_{ijkl}+\Sch_{ik}g_{jl}-\Sch_{jk}g_{il}-\Sch_{il}g_{jk}+\Sch_{jl}g_{ik}.
\end{equation*}
The Cotton and Bach tensors are
\begin{equation*}
    C_{ijk}=\nabla^g_k\Sch_{ij}-\nabla^g_j\Sch_{ik}, \quad B_{ij}=\nabla^g_kC_{ij}^{\;\;\;k}-\Sch^{kl}W_{kijl}.
\end{equation*}
Latin indices $i, j, k$ run between $1$ and $n$ in local coordinates, or can be interpreted as labels for $TM$ or its dual in invariant expressions such as those above.

We will denote by $\Sigma$ an immersed submanifold of $(M,g)$ of dimension $k$, $1\leq k\leq n-1$. We use $\alpha, \beta, \gamma$ as index labels for $T\Sigma$ and $\alpha', \beta', \gamma'$ for the normal bundle $N\Sigma$. A Latin index $i$ thus specializes either to an $\alpha$ or an $\alpha'$. The restriction of the metric $g_{ij}$ to $\Sigma$ can be identified with the metric $g_{\alpha\beta}$ induced on $\Sigma$ together with the bundle metric $g_{\alpha'\beta'}$ induced on $N\Sigma$.

The \textit{second fundamental form} $L:S^2T\Sigma\to N\Sigma$ is defined by $L(X,Y)=\left(\nabla^g_X Y\right)^{\perp}$. We typically write it as $L^{\alpha'}_{\alpha\beta}$. The mean curvature vector is $H=\frac{1}{k}\tr_g L$.

Under a conformal change of \textbf{the ambient metric} $\tilde{g}=e^{2\phi}g$. As \textit{vector-valued} functions, extrinsic geometric quantities change as follows:
\begin{equation}\label{transformation-vector}
    \cL_{\alpha\beta}=L_{\alpha\beta}-\left(\nabla^g \phi\right)^{\perp}g_{\alpha\beta}, \quad \cH=e^{-2\phi}(H-\left(\nabla^g \phi\right)^{\perp}), \quad \mathring{\cL}_{\alpha\beta}=\mathring{L}_{\alpha\beta}.
\end{equation}
%Here we emphasize that given a point $x_0\in \Sigma$, let $(x_1, \cdots, x_n)$ be a local coordinate system on a neighborhood $U_{x_0}^{M}$ in $M$ with associated coordinate vector fields given by $X_i=\frac{\partial}{\partial x_i}$. We can find local coordinates on $U_{x_0}^{M}$ such that
%\begin{equation*}
 %   \Sigma\cap U_{x_0}^{M}=\left\{(x_1, \cdots, x_n)\in U_{x_0}^{M}: x_{k+1}=0, \cdots, x_n=0\right\}.
%\end{equation*}
%It is clear that $\{X_i\}_{i=1}^n$ are a basis for the tangent bundle. After the conformal change, all covariant derivatives are calculated with respect to this basis instead of their normalized basis.
In components, given an orthonormal basis $\{N_{\alpha'}\}_{\alpha'=1}^{n-k}$ for the normal bundle, $\{N_{\alpha'}\}_{\alpha'=1}^{n-k}$ and above quantities change as follows
\begin{equation}\label{transformation-component}
\begin{split}
    \cN_{\alpha'}=e^{-\phi}N_{\alpha'}, &\quad \cL^{\alpha'}_{\alpha\beta}=e^{\phi}\left(L^{\alpha'}_{\alpha\beta}-\left(\nabla^g_{\alpha'} \phi\right)^{\perp}g_{\alpha\beta}\right),\\  \mathring{\cL}^{\alpha'}_{\alpha\beta}=e^{\phi}\mathring{L}^{\alpha'}_{\alpha\beta}, &\quad \cH^{\alpha'}=e^{-\phi}(H^{\alpha'}-\left(\nabla_{\alpha'}^g \phi\right)^{\perp}).
    \end{split}
\end{equation}
From (\ref{transformation-vector}), it is clear that
\begin{equation*}
\begin{split}
|\mathring{\cL}|_{\cg}^2&=\mathring{\cL}^{\alpha'}_{\alpha\beta}\mathring{\cL}^{\beta'}_{\delta\gamma}\cg^{\alpha\delta}\cg^{\beta\gamma}\cg_{\alpha' \beta'}=e^{-2\phi}\mathring{L}^{\alpha'}_{\alpha\beta}\mathring{L}^{\beta'}_{\delta\gamma}g^{\alpha\delta}g^{\beta\gamma}g_{\alpha' \beta'}=e^{-2\phi}|\mathring{L}|_{g}^2,\\
\left(\mathring{\cL}^2\right)_{\alpha\beta}^{\alpha'}&=\left(\mathring{\cL}_{\alpha\delta}\mathring{\cL}_{\beta\gamma}\cg^{\delta\gamma}\right)^{\alpha'}=e^{\phi}\left(e^{-2\phi}\mathring{L}_{\alpha\delta}\mathring{L}_{\beta\gamma}g^{\delta\gamma}\right)^{\alpha'}=e^{-\phi}\left(\mathring{L}^2\right)_{\alpha\beta}^{\alpha'}
\end{split}
\end{equation*}
so $|\mathring{\cL}|_{\cg}^k \dvol_{\cg}=|\mathring{L}|_{g}^k \dvol_{g}$ and $|\mathring{\cL}^2|_{\cg}^{\frac{k}{2}}\dvol_{\cg}=|\mathring{L}^2|_{g}^{\frac{k}{2}}\dvol_{g}$ are global conformal invariants.
%\begin{remark}
 %   The conformal transformation laws stated in \cite{Mondino-Nguyen} are very confusing because \cite[(12)]{Mondino-Nguyen} is the conformal transformation law of components but geometric quantities in it are vector-valued. 
%\end{remark}

For embeddings into $\bR^5$, Mondino and Nyugen's result\cite[Theorem 4.1]{Mondino-Nguyen18} confirms that on hypersurface $\Sigma^4$, the space of global conformal invariants is four-dimensional and spanned by 
\begin{equation}\label{basis-of-global-conformal-invariants}
 \mE_{GR},\quad   \int_{\Sigma} |\mathring{L}|^4\dvol,\quad  \int_{\Sigma}|\mathring{L}^2|^2 \dvol, \quad \chi(\Sigma^4).
\end{equation}
While for embeddings of $\Sigma^4$ into $\bR^n$ with $n \geq 6$, the space of global conformal invariants is six-dimensional: a basis is $\mE_{GR}$, the invariant $\mathcal{I}$ defined in \cite{CGKTW25} and four contractions of $\mathring{L}$ which are obtained by breaking $\mathring{L}_{\alpha\beta}^{\alpha'}\mathring{L}_{\gamma\delta}^{\beta'}$ into symmetric and antisymmetric-in-normal pieces and then taking possible contractions; see \cite[Section 6]{CGKTW25} for more details. In this paper, we will merely consider global conformal invariants listed in (\ref{basis-of-global-conformal-invariants}).

\subsection{Extrinsic GJMS operators and $Q-$curvatures}
In this subsection, we recall the extrinsic Paneitz operator and its $Q_4$ curvature constructed in \cite{Case-Graham-Kuo}.

\begin{theorem}[\cite{Case-Graham-Kuo}]
Let $(M^n,g)$, $n\geq 5$ be a Riemannian manifold and $\Sigma^k\subset M^n$, $4\leq k\leq n-1$ be a smoothly immersed submanifold. The extrinsic GJMS operator $P_4$ on $\Sigma^k$ is given by
    \begin{equation}
P_4=\Delta^2+\nabla^{\alpha}\left(T_{\alpha\beta}\nabla^{\beta}\right)+\frac{k-4}{2}Q_4,
    \end{equation}
where 
\begin{equation}
    \begin{split}
T_{\alpha\beta}=&4\Sch_{\alpha\beta}+4H^{\alpha'}L_{\alpha\beta\alpha'}-\left[(k-2)\Sch_{\gamma}^{\;\;\gamma}+\frac{1}{2}(k^2-2k+4)|H|^2\right]g_{\alpha\beta},\\
Q_4=&-\Delta \left(\Sch_{\alpha}^{\;\;\alpha}+\frac{k}{2}|H|^2\right)-2\big|\Sch_{\alpha\beta}+H^{\alpha'}L_{\alpha\beta\alpha'}-\frac{1}{2}|H|^2g_{\alpha\beta}\big|^2\\
&+2\big|\Sch_{\alpha\alpha'}-\nabla_{\alpha}H_{\alpha'}\big|^2+\frac{k}{2}\left(\Sch_{\alpha}^{\;\;\alpha}+\frac{k}{2}|H|^2\right)^2\\
&-W^{\alpha}_{\;\;\;\alpha'\alpha\beta'}H^{\alpha'}H^{\beta'}-4C^{\alpha}_{\;\; \alpha\alpha'}H^{\alpha'}-\frac{2}{n-4}B^{\alpha}_{\;\;\alpha}.
\end{split}
\end{equation}
\end{theorem}

In particular, for a four-dimensional submanifold $\Sigma^4$ immersed into $\bR^n$, we have
\begin{equation}
        P_4=\Delta^2+\nabla^{\alpha}\left(T_{\alpha\beta}\nabla^{\beta}\right),
    \end{equation}
where 
\begin{equation}
    \begin{split}
T_{\alpha\beta}=&4H^{\alpha'}L_{\alpha\beta\alpha'}-6|H|^2 g_{\alpha\beta},\\
Q_4=&-2\Delta |H|^2-2\big|H^{\alpha'}L_{\alpha\beta\alpha'}-\frac{1}{2}|H|^2g_{\alpha\beta}\big|^2+2\big|\nabla_{\alpha}H_{\alpha'}\big|^2+8|H|^4.
\end{split}
\end{equation}
Moreover, under the conformal change $\hat{g}=e^{2\omega}g$, $P_4$ satisfies
\begin{equation}\label{Trans of Q}
e^{4\omega|_{\Sigma}}\hat{Q}_4=Q_4+P_4\left(\omega|_{\Sigma}\right).
\end{equation}
Recall that the fourth order Willmore energy introduced by Graham and Reichert is given by
\begin{equation*}
    \mE_{GR}(\Phi):=\int_{\Sigma} \left(\big|\nabla_{\alpha}H_{\alpha'}\big|^2-\big|H^{\alpha'}L_{\alpha\beta\alpha'}\big|^2+7|H|^4\right)\dvol,
\end{equation*}
or
\begin{equation*}
    \mE_{GR}(\Phi)=\int_{\Sigma} \left(\big|\nabla_{\alpha}H_{\alpha'}\big|^2-\big|H^{\alpha'}\mathring{L}_{\alpha\beta\alpha'}\big|^2+3|H|^4\right)\dvol,
\end{equation*}
where $\mathring{L}_{\alpha\beta\alpha'}$ is the trace-free part of the second fundamental form. By the definition of $Q_4$, on closed submanifold $\Sigma^4$, we have
\begin{equation*}
    \int_{\Sigma} Q_4\dvol=2\mE_{GR}.
\end{equation*}    
\subsection{Spherical harmonics}
In  this subsection, we recall some basic facts concerning spherical harmonics, which will be used in section \ref{connected-sum}. Let $\mathcal{R}$ denote the space of real polynomials on $\bR^{n+1}$ and $\mathcal{H}$ be the subspace of harmonic polynomials. For all $h\in \bN_0$, let $\mathcal{R}_h\subset \mathcal{R}$ denote the space of homogeneous polynomials of degree $h$ and set $\mathcal{H}_h=\mathcal{R}_h\cap \mathcal{H}$. Clearly, we have the following decomposition
\begin{equation*}
    \mathcal{R}=\bigoplus_{h\in \bN_0}\mathcal{R}_h, \quad \quad \mathcal{H}=\bigoplus_{h\in \bN_0}\mathcal{H}_h.
\end{equation*}
It is well-known that 
\begin{equation}
    \mathcal{R}_h=\mathcal{H}_h\oplus |x|^2\mathcal{R}_{h-2} \quad \quad ~~\mbox{for all}~~ h\in \bN_0,
\end{equation}
where $x=(x_1,\cdots, x_{n+1})\in \bR^{n+1}$. 
By the Stone--Weierstrass theorem, we know that the space $L^2(S^{n})$ endowed with the inner product
\begin{equation*}
    (F,G)=\int_{S^{n}}FG\dvol_{g_c},
\end{equation*}
can be decomposed as 
\begin{equation*}
    L^2(S^{n})=\overline{\bigoplus_{h\in\bN_0}\mathcal{H}^{\bS}_{h}},
\end{equation*} 
where $\mathcal{H}^{\bS}_{h}$ is the restriction of $\mathcal{H}_{h}$ to the sphere. 
Through this paper, the superscript $\bS$ will be used to denote the sets of restrictions to $\mathbb{S}^{n}$. The dimension of $\mathcal{H}^{\bS}_{h}$ is 
\begin{equation*}
    \mathrm{dim} (\mathcal{H}^{\bS}_{h})=N_{h}:=\binom{n+h}{n}-\binom{n+h-2}{n}.
\end{equation*}
In particular, in the sequel, we denote $\{X_i\}_{i=1}^{N_1}$, $\{Y_j\}_{j=1}^{N_2}$ and $\{W_k\}_{k=1}^{N_3}$ normalized basis in $\mathcal{H}^{\bS}_{1}$, $\mathcal{H}^{\bS}_{2}$ and $\mathcal{H}^{\bS}_{3}$ on $\mathbb{S}^3$, respectively, where $ N_1=4, \, N_2=9, \, N_3=16.$

\section{Energy identity under inversion}\label{energy-identity}
The main purpose of this section is to prove the energy identity (\ref{loss of energy}) and demonstrate its geometric interpretation. In the sequel, we assume that $\Sigma^4$ is a closed immersed submanifold in $\bR^n$ given by $\Phi: \Sigma^4\hookrightarrow \bR^n$. 
\subsection{Energy identity}
Without loss of generality, we may assume that $0\in \Phi(\Sigma)$. We consider the inversion centered at the origin of radius $1$: $I: x\to \frac{x}{|x|^2}$. Using the formula
\begin{equation*}
    DI(X)=\frac{1}{|X|^2}\left(\Id_{\bR^n}-2\frac{X}{|X|}\otimes \frac{X}{|X|}\right), ~~\mbox{for}~~ X\neq 0,
\end{equation*}
direct calculation yields
\begin{equation*}
    I^{*}g_0=\frac{1}{|X|^4}g_0,
\end{equation*}
which implies that
\begin{equation*}
    I: (\bR^n\backslash \{0\}, I^{*}g_0)\to (\bR^n\backslash \{0\}, g_0)
\end{equation*}
is an isometry.

Besides, composing with $\Phi$, we obtain an immersion of the manifold $\hat{\Sigma}=\Sigma\backslash \Phi^{-1}\{0\}$ into $\bR^n$ given by $I\circ \Phi: \hat{\Sigma}\hookrightarrow \bR^n$ with induced metric $\hat{g}=\frac{1}{|\Phi|^4}g$, i.e. $\omega=-\ln |\Phi|^2$.

\begin{proof}[Proof of Theorem \ref{Main 2}]
Recall that the inversion establishes an isometry between the following two pairs
\begin{equation*}
    I: \hat{\Sigma}\xhookrightarrow{\Phi}(\bR^n\backslash \{0\}, I^{*}g_0) \to \hat{\Sigma}\xhookrightarrow{I\circ \Phi}(\bR^n\backslash \{0\}, g_0).
\end{equation*}
Therefore, from the transformation law (\ref{Trans of Q}) for $Q_4$, we know that
\begin{equation*}
    -2\int_{\hat{\Sigma}}\hat{\Delta}|\hat{H}|^2\dvol_{\hat{g}}+2\mE_{GR}(\hat{\Sigma})=2\mE_{GR}(\Sigma)+\int_{\Sigma}P_4 \left(-\ln |\Phi|^2\right)\dvol_g.
\end{equation*}

We assume that $\Phi$ is a graph on $\mathbb{B}_R(0)\subset \bR^4$, i.e. $\Phi=(x, \phi)$, where $x\in \bR^4$, $\phi\in \bR^{n-4}$ and $\phi(0)=0$, $D\phi(0)=0$. The induced metric $g=(g_{\alpha\beta})$ is given by $\Id_{\bR^4}+D^{T}\phi\cdot D\phi$ and the outward unit normal vector on $\partial \mathbb{B}_r$, $0<r<R$, is defined by
\begin{equation*}
    \eta_r=\frac{g^{-1}e_r}{\langle g^{-1}e_r, e_r\rangle ^{\frac{1}{2}}}, \quad e_r=\frac{x}{r}.
\end{equation*}

Integration by parts yields
\begin{equation}
    \int_{\Sigma \backslash \mathbb{B}_r}P_4 \left(-\ln |\Phi|^2\right)\dvol_g=\int_{\partial \mathbb{B}_r}\left(\langle \eta_r, \nabla \Delta \left(\ln |\Phi|^2\right)\rangle+T\left(\nabla \ln |\Phi|^2, \eta_r\right)\right)dA_g.
\end{equation}
By definitions, we know that
\begin{equation*}
    L_{\alpha\beta}=(0, D^2_{\alpha\beta}\phi)^{\perp}, \quad H=\frac{1}{4}(0, g^{\alpha\beta}D^2_{\alpha\beta}\phi)^{\perp}, \quad \nabla \ln |\Phi|^2=\frac{2g^{-1}\left(x+D^{T}\phi\cdot \phi\right)}{|x|^2+|\phi|^2}.
\end{equation*}
Besides, by our assumptions, in $\mathbb{B}_r(0)$, 
\begin{equation*}
    g=\Id_{\bR^4}+O(r^2).
\end{equation*}
In particular, the induced metric on $\partial \mathbb{B}_r$ ($=\mathbb{S}^3$ in topology) converges to the canonical metric $g_c$ uniformly as $r$ goes to $0$. Therefore, we know that on $\partial \mathbb{B}_r$,
\begin{equation}\label{area}
    \eta_r=\frac{x}{r}+O(r), \quad dA_g=r^3dA_{S^3}+o(r^3), \quad \nabla \ln |\Phi|^2=\frac{2x}{r^2}+O(r).
\end{equation}
Note that $T_{\alpha\beta}=4H^{\alpha'}L_{\alpha\beta\alpha'}-6|H|^2 g_{\alpha\beta}=O(1)$. Combining them together yields
\begin{equation*}
      T\left(\nabla \ln |\Phi|^2, \eta_r\right)dA_g=O(1)r^2dA_{\mathbb{S}^3}+o(r^2),
\end{equation*}
which implies that $\int_{\partial \mathbb{B}_r}T\left(\nabla \ln |\Phi|^2, \eta_r\right)dA_g$ converges to $0$ uniformly as $r$ goes to $0$.

Besides, from the expansion of the metric, we know that 
\begin{equation*}
    \Delta \ln |\Phi|^2=\delta^{\alpha\beta}D^2_{\alpha\beta}\ln |\Phi|^2+O(1).
\end{equation*}
Direct calculation yields
\begin{equation*}
    \delta^{\alpha\beta}D^2_{\alpha\beta}\ln |\Phi|^2=2\frac{4+|D\phi|^2_{g_e}+\phi\cdot \Delta_{g_e}\phi}{|x|^2+|\phi|^2}-4\frac{|x|^2+\langle D|\phi|^2, x\rangle+|\phi|^2|D\phi|_{g_e}^2}{\left(|x|^2+|\phi|^2\right)^2}.
\end{equation*}
It's not hard to see that
\begin{equation*}
    |D\phi|^2_{g_e}=O(r^2), \quad \phi\cdot \Delta_{g_e}\phi=O(r^2), \quad \langle D|\phi|^2, x\rangle=O(r^4), \quad |\phi|^2|D\phi|_{g_e}^2=O(r^6).
\end{equation*}
Therefore, similar to above, we have
\begin{equation*}
    \int_{\partial \mathbb{B}_r}\langle \eta_r, \nabla \Delta \left(\ln |\Phi|^2\right)\rangle dA_g= -8\int_{\partial \mathbb{B}_r}\left\langle \frac{x}{r}, \frac{x}{r^4}\right\rangle r^3dA_{\mathbb{S}^3}+O(r)\to -8\Vol(\mathbb{S}^3)
\end{equation*}
as $r$ goes to $0$.
In sum, we have
\begin{equation*}
    \int_{\Sigma \backslash \mathbb{B}_r}P_4 \left(-\ln |\Phi|^2\right)\dvol_g\to -8\Vol(\mathbb{S}^3).
\end{equation*}
Similarly, integration by parts yields
\begin{equation*}
    -\int_{\Sigma \backslash \mathbb{B}_r}\hat{\Delta}|\hat{H}|^2\dvol_{\hat{g}}=\int_{\partial \mathbb{B}_r} \langle \hat{\eta}, \hat{\nabla} |\hat{H}|^2\rangle_{\hat{g}}dA_{\hat{g}}.
\end{equation*}
By conformal transformation law, we know that
\begin{equation*}
    |\hat{H}|^2_{\hat{g}}=|\Phi|^4\left(H+8\frac{\Phi^{\perp}}{|\Phi|^2}\right)^2,
\end{equation*}
where $\Phi^{\perp}$ is the normal part of $\Phi$.
Note that 
\begin{equation*}
    H+8\frac{\Phi^{\perp}}{|\Phi|^2}=O(1).
\end{equation*}
Therefore, we have
\begin{equation*}
\begin{split}
    \langle \hat{\eta}, \hat{\nabla} |\hat{H}|^2\rangle_{\hat{g}}dA_{\hat{g}}&=\eta(|\hat{H}|^2)|\Phi|^2\cdot |\Phi|^{-6}dA_{g}\\
   &= \eta \left(H+8\frac{\Phi^{\perp}}{|\Phi|^2}\right)^2 dA_g\\
    &+\left(H+8\frac{\Phi^{\perp}}{|\Phi|^2}\right)^2\langle \eta, \nabla |\Phi|^4\rangle_g |\Phi|^{-4}dA_g 
    \end{split}
\end{equation*}
Since 
\begin{equation*}
    \langle \eta, \nabla |\Phi|^4\rangle_g |\Phi|^{-4}r^3=O(r),
\end{equation*}
we conclude that 
\begin{equation*}
    \int_{\Sigma \backslash \mathbb{B}_r}\hat{\Delta}|\hat{H}|^2 \dvol_{\hat{g}}\to 0
\end{equation*}
as $r$ goes to $0$.

Note that $\Vol(\mathbb{S}^3)=2\pi^2$ and $\Vol(\mathbb{S}^4)=\frac{8}{3}\pi^2$. Finally, we have
\begin{equation*}
\begin{split}
    \mE_{GR}(\hat{\Sigma})&=\mE_{GR}(\Sigma)-3\Vol(\mathbb{S}^4)\cdot \mathrm{card} \Phi^{-1}(0)\\
    &=\mE_{GR}(\Sigma)-\mE_{GR}(\mathbb{S}^4)\cdot \mathrm{card} \Phi^{-1}(0).
    \end{split}
\end{equation*}
Since $|\mathring{L}|^4$, $|\mathring{L}^2|^2$ are two pointwise conformal invariants, we can conclude that the loss of the energy under the inversion is independent of $\int |\mathring{L}|^4$ and $\int |\mathring{L}^2|^2$. Hence,
\begin{equation*}
    \mE^{(\mu, \nu)}(\hat{\Phi})=\mE^{(\mu, \nu)}(\Phi)-\mE^{(\mu, \nu)}(\mathbb{S}^4)\cdot \mathrm{card} \Phi^{-1}(0).
\end{equation*}
\end{proof}

\subsection{Geometric interpretation of (\ref{loss of energy})}
In this subsection, we explain a little bit about the geometric picture behind the identity \ref{loss of energy}. Recall that on $(\Sigma^4, g)$, the Branson's $\mQ_4$ curvature of $g$ is defined as
\begin{equation*}
    \mQ_4:=\frac{1}{6}\left\{-\Delta R_g+\frac{1}{4}R^2_g-3|\mathring{\Ric}|_g^2\right\},
\end{equation*}
where $\mathring{\Ric}$ is the traceless part of $\Ric$. It is well known that the $\mQ_4$ curvature is an integral conformal invariant associated to the fourth order intrinsic Paneitz operator $\mP_4$
\begin{equation*}
    \mP_4:=\Delta^2+\nabla^{\alpha}\left(\left(\frac{2}{3}R g_{\alpha\beta}-2\Ric_{\alpha\beta}\right)\nabla^{\beta}\right).
\end{equation*}
Moreover, under the conformal change $\hat{g}=e^{2\omega}g$, $\mP_4$ satisfies
\begin{equation*}
e^{4\omega}\hat{\mQ}_4=\mQ_4+\mP_4\left(\omega\right).
\end{equation*}
By Gauss--Codazzi equations, we have
\begin{equation*}
    \begin{split}
        \Ric_{\alpha\beta}=-L^{\alpha'}_{\alpha\gamma}L^{\gamma}_{\beta\alpha'}&+4H^{\alpha'}L_{\alpha\beta\alpha'}=-\mathring{L}^{\alpha'}_{\alpha\gamma}\mathring{L}^{\gamma}_{\beta\alpha'}+2H^{\alpha'}\mathring{L}_{\alpha\beta\alpha'}+3|H|^2g_{\alpha\beta},\\
        R&=16|H|^2-|L|^2=12|H|^2-|\mathring{L}|^2.
    \end{split}
\end{equation*}
Therefore, 
\begin{equation*}
    \frac{2}{3}R g_{\alpha\beta}-2\Ric_{\alpha\beta}=2L^{\alpha'}_{\alpha\gamma}L^{\gamma}_{\beta\alpha'}-8H^{\alpha'}L_{\alpha\beta\alpha'}+\frac{2}{3} \left(16|H|^2-|L|^2\right)g_{\alpha\beta}.
\end{equation*}
Note that the extrinsic Paneitz operator $P_4$ and the intrinsic Paneitz operator $\mP_4$ have the same leading term $\Delta^2$. By \cite[Theorem 5.3]{Case-Graham-Kuo}, the difference $P_4-\mP_4$ is given by
\begin{equation*}
    P_4-\mP_4=\nabla^{\alpha}\left(S_{\alpha\beta}\nabla^{\beta}\right),
\end{equation*}
where 
\begin{equation*}
    \begin{split}
        S_{\alpha\beta}&=T_{\alpha\beta}-\frac{2}{3}R g_{\alpha\beta}+2\Ric_{\alpha\beta}\\
        &=12H^{\alpha'}L_{\alpha\beta\alpha'}-2L^{\alpha'}_{\alpha\gamma}L^{\gamma}_{\beta\alpha'}-\frac{2}{3} \left(25|H|^2-|L|^2\right)g_{\alpha\beta},
    \end{split}
\end{equation*}
and $Q_4-\mQ_4$ satisfies
\begin{equation*}
    e^{4\omega}\left(\hat{Q}_4-\hat{\mQ}_4\right)=\left({Q}_4-{\mQ}_4\right)+ \nabla^{\alpha}\left(S_{\alpha\beta}\nabla^{\beta}\omega\right).
\end{equation*}
Direct calculation shows that
\begin{equation*}
\begin{split}
    \frac{1}{2}\int_{\Sigma}\left({Q}_4-{\mQ}_4\right)\dvol_g=&\int_{\Sigma} \left(\big|\nabla_{\alpha}H_{\alpha'}\big|^2+\frac{1}{2}|H|^2|\mathring{L}|^2-H^{\alpha'}\tr_g\left(\mathring{L}^3_{\alpha'}\right)\right)\dvol_g\\
    &-\frac{1}{12}\int_{\Sigma}|\mathring{L}|^4 \dvol_g+\frac{1}{4}\int_{\Sigma}\tr_g \mathring{L}^4\dvol_g.
    \end{split}
\end{equation*}
%\textcolor{blue}{Here we would like to emphasize that in general, $|H|^2|\mathring{L}|^2\neq \big|H^{\alpha'}\mathring{L}_{\alpha\beta\alpha'}\big|^2$. $|H|^2|\mathring{L}|^2$ comes from $R^2_g$ and $\big|H^{\alpha'}\mathring{L}_{\alpha\beta\alpha'}\big|^2$ are always canceled.}

From the proof of Theorem \ref{Main 2} and above discussion, we can conclude that the functional
\begin{equation}\label{new-conformal-invariant}
    \int_{\Sigma} \left(\big|\nabla_{\alpha}H_{\alpha'}\big|^2+\frac{1}{2}|H|^2|\mathring{L}|^2-H^{\alpha'}\tr_g\left(\mathring{L}^3_{\alpha'}\right)\right)\dvol_g
\end{equation}
is invariant under translations, rotations, dilations and inversions on $\bR^n$.

Moreover, recall that by the Gauss--Bonnet--Chern formula on closed four dimensional manifolds, we have
\begin{equation*}
    \int_{\Sigma}\mQ_4 \dvol_g=8\pi^2\chi(\Sigma^4)-\int_{\Sigma}\frac{|W|^2_g}{4}\dvol_g.
\end{equation*}
Combining
\begin{equation*}
    \int_{\Sigma} Q_4\dvol=-2\int_{\Sigma}\Delta |H|^2\dvol+2\mE_{GR}=2\mE_{GR},
\end{equation*} 
yields
\begin{equation*}
    \mE_{GR}=4\pi^2\chi(\Sigma^4)+\frac{1}{2}\int_{\Sigma}\left({Q}_4-{\mQ}_4\right)\dvol_g-\int_{\Sigma}\frac{|W|^2_g}{8}\dvol_g.
\end{equation*}
We already know that 
\begin{equation*}
    \frac{1}{2}\int_{\Sigma}\left({Q}_4-{\mQ}_4\right)\dvol_g-\int_{\Sigma}\frac{|W|^2_g}{8}\dvol_g
\end{equation*}
is invariant under inversions. It implies that the loss of the energy illustrated in Theorem \ref{Main 2} is totally due to the change of the Euler characteristic. More precisely, we consider the long exact sequence of the relative homology for the pair $(\Sigma^4, \Sigma^4\backslash \Phi^{-1}(0))$. By the excision theorem, we know that
\begin{equation*}
    H_k (\Sigma^4, \Sigma^4\backslash \Phi^{-1}(0))\cong \left(H_{k-1}(S^3)\right)^{\otimes m}=\left\{
    \begin{array}{cc}
        \bZ^m & k=4,  \\
        0 & ~~\mbox{otherwise},
    \end{array}
    \right. 
\end{equation*}
where $m=\mathrm{card} \Phi^{-1}(0)$. Therefore, 
\begin{equation*}
    \chi(\Sigma^4\backslash \Phi^{-1}(0)))=\chi(\Sigma^4)-m.
\end{equation*}

The Gauss--Bonnet--Chern formula (\ref{complete-GBC}) in Lemma \ref{Gauss--Bonnet--Chern--formula} yields the desired. In conclusion, we have

\begin{equation*}
\begin{split}
    \mE_{GR}(\Sigma^4\backslash G^{-1}(0))=&\frac{1}{2}\int_{\Sigma^4\backslash G^{-1}(0)} Q_4\dvol\\
    =&\frac{1}{2}\int_{\Sigma^4\backslash G^{-1}(0)}\left({Q}_4-{\mQ}_4\right)\dvol+\frac{1}{2}\int_{\Sigma^4\backslash G^{-1}(0)}{\mQ}_4\dvol\\
    =&\frac{1}{2}\int_{\Sigma^4}\left({Q}_4-{\mQ}_4\right)\dvol-\int_{\Sigma^4}\frac{|W|^2_h}{8}\dvol\\
    &+4\pi^2\left(\chi(\Sigma^4)-2m\right)\\
    =&\frac{1}{2}\int_{\Sigma^4}\left({Q}_4-{\mQ}_4\right)\dvol+\frac{1}{2}\int_{\Sigma^4}{\mQ}_4\dvol-8\pi^2m\\
    =&\mE_{GR}(\Sigma^4)-8\pi^2 m.
    \end{split}
\end{equation*}

\begin{remark}
In the hypersurface case, there is an alternative way to verify the invariance of (\ref{new-conformal-invariant}) under inversions. Note that for ${\Sigma^4\hookrightarrow \bR^5}$, we have
\begin{equation*}
    \det L=H^4-\frac{1}{2}H^2|\mathring{L}|^2_g+\frac{1}{3}H\left(\tr_g \mathring{L}^3\right)+\det \mathring{L},
\end{equation*}
where we use the relation
\begin{equation*}
    \det \mathring{L}=\frac{1}{8}|\mathring{L}|^4_g-\frac{1}{4}\left(\tr_g \mathring{L}^4\right).
\end{equation*}
Therefore,
\begin{equation}\label{decomposition of energy}
\begin{split}
    \mE_{GR}&=3\int_{\Sigma}\det L \dvol_g-3\int_{\Sigma}\det \mathring{L} \dvol_g\\
    &+\int_{\Sigma} \left(\big|\nabla H\big|^2+\frac{1}{2}\big|H|^2|\mathring{L}\big|^2-H\left(\tr_g \mathring{L}^3\right)\right)\dvol_g.
    \end{split}
\end{equation}
Besides, on closed manifolds, we have
\begin{equation}\label{CGB}
    \frac{4\pi^2}{3}\chi(\Sigma)=\int_{\Sigma}\det L \dvol_g.
\end{equation}

Next, we consider the conformal Gauss map $Y_{\Phi}$ of $\Phi$ which induces a new metric $\kg$ on $\Sigma$ given by
\begin{equation*}
    \kg_{ij}:=\mathring{L}_i^{\;\;k}\mathring{L}_{kj}.
\end{equation*}
Here we need to emphasize that $\mathring{L}_{ij}$ in above is given by $\langle \mathring{L}_{ij}, N\rangle $ for a given unit normal vector $N$ with respect to $g$ (see \cite[(26) and Lemma 3.1]{Martino23} for precise definition).
In particular, under the conformal change $\hat{g}=e^{2\phi}g$, $\kg$ is invariant because by conformal transformation laws,
\begin{equation*}
    \hat{\kg}_{ij}=\hat{\mathring{L}}_{ik}\hat{\mathring{L}}_{lj}\hat{g}^{kl}=e^{w}e^{w}e^{-2w} \kg_{ij}=\kg_{ij}.
\end{equation*}
Furthermore, if we assume that $\mathring{L}_{ij}\neq 0$ on $\Sigma$ and $\det \mathring{L}>0$, by \cite[Prop Theorem II]{Martino23}, we know that
\begin{equation*}
\begin{split}
    &\int_{\Sigma} \left(\big|\nabla H\big|^2+\frac{1}{2}\big|H|^2|\mathring{L}\big|^2-H\left(\tr_g \mathring{L}^3\right)\right)\dvol_g-3\int_{\Sigma}\det \mathring{L} \dvol_g\\
    =&3\Vol_{\kg}(\Sigma)-\frac{1}{2}\int_{\Sigma} R_{\kg}\dvol_{\kg},
    \end{split}
\end{equation*}
where we use the relation $\int_{\Sigma}\det \mathring{L} \dvol_g=\Vol_{\kg}(\Sigma)$. The desired invariance follows easily.
\end{remark}

The following GBC formula is a special case of the main result in \cite[Theorem 1.5]{Li-Wang24}. For self-completeness, we sketch the proof in the same spirit of \cite{Chang-Qing-Yang00}.

\begin{lemma}[GBC formula on $\hat{\Sigma}$]\label{Gauss--Bonnet--Chern--formula}
    For the manifold $\hat{\Sigma}=\Sigma\backslash \Phi^{-1}\{0\}$ immersed in $\bR^n$ given by $I\circ \Phi: \hat{\Sigma}\hookrightarrow \bR^n$ with induced metric $\hat{g}=\frac{1}{|\Phi|^4}g$, we have
    \begin{equation}\label{complete-GBC}
    \int_{\Sigma}\mQ_4 \dvol_g=8\pi^2\left(\chi(\hat{\Sigma})-m\right)-\int_{\Sigma}\frac{|W|^2_g}{4}\dvol_g.
\end{equation}
\end{lemma}

\begin{proof}
    The proof follows from \cite[Section 2]{Chang-Qing-Yang00}. We assume that $\Phi$ is a graph on $B_R(0)\subset \bR^4$, i.e. $\Phi=(x, \phi)$, where $x\in \bR^4$, $\phi\in \bR^{n-4}$ and $\phi(0)=0$, $D\phi(0)=0$. By our assumptions, in $B_{r}(0)$, the metric is
\begin{equation*}
    g=\Id_{\bR^4}+O(r^2).
\end{equation*}
Therefore, around each point $p_i\in \Phi^{-1}(0)$, we can find a neighborhood $U_i\subset \hat{\Sigma}$ corresponding to $B_r(0)\backslash \{0\}$ with the metric $\frac{g}{|\Phi|^4}$. We define $v=\omega+s$, where $\omega=-\ln |\Phi|^2$ and $s=\ln r$ for $r=|x|$, $x\in B_r(0)$. Hence, $v$ is a radial function asymptotically. Moreover, non-radial part in the metric $\frac{g}{|\Phi|^4}$ has higher order, which has no contribution on the isoperimetric constant. For convenience, we ignore these higher order terms and consider the radial contribution from $v$ merely. Without loss of generality, we may assume that $v$ is a radial function in $B_r(0)$.

We set $t=-s$ so that as $s\to 0^{+}$, $t\to 0^{-}$. Similar to \cite[(2.16)]{Chang-Qing-Yang00}, we can obtain
\begin{equation}
    8\pi^2\chi(\hat{\Sigma})-\int_{\Sigma}\frac{|W|^2_g}{4}\dvol_g-\int_{\Sigma\backslash \Phi^{-1}(B_r(0))}\mQ_4 \dvol_g=m\cdot \Vol(S^3)\left(-v^{'''}(t)+4v^{'}(t)\right),
\end{equation}
where $'=\frac{d}{dt}$.

First, from above discussion, it is easy to see that $\mQ_4$ is absolutely integrable on $\hat{\Sigma}$. One crucial step in \cite[Section 2]{Chang-Qing-Yang00} is to show that the globally defined radial function $v$ on $\bR^4$ there has the form of
\begin{equation*}
    v(t)=c_0+c_1t+c_2e^{-2t}+c_3e^{3t}+f(t),
\end{equation*}
for some constants $c_1$, $c_2$, $c_3$, $c_4$ and a function $f(t)$ depending on $\int_{\hat{\Sigma}}|\mQ_4|\dvol$. On other hand, the absolute integrability of $\mQ_4$ ensures that $c_2=0$ (see \cite[Lemma 2.3]{Chang-Qing-Yang00} for more details) and the additional assumption on the non-negative scalar curvature ensures that $c_3=0$ (see \cite[Lemma 2.5]{Chang-Qing-Yang00}), i.e we can conclude that
\begin{equation*}
    \lim_{t\to \infty}v^{''}(t)=\lim_{t\to \infty}v^{'''}(t)=0.
\end{equation*}
In our setting, since $v=ct$ asymptotically, we can immediately conclude that $c_2=c_3=0$ and
\begin{equation*}
    \lim_{t\to \infty} v^{'}(t)=\frac{1}{8\pi^2\cdot m}\left(8\pi^2\chi(\hat{\Sigma})-\int_{\Sigma}\frac{|W|^2_g}{4}\dvol_g-\int_{\hat{\Sigma}}\mQ_4 \dvol_g\right).
\end{equation*}
By the completeness of the metric $\frac{g}{|\Phi|^4}$, we conclude that $\lim_{t\to \infty} v^{'}(t)\geq 0$. As $g$ converges to the standard metric uniformly as $r$ goes to zero, the desired Gauss--Bonnet--Chern formula (\ref{complete-GBC}) follows immediately.
\end{proof}

\section{Expansion at infinity}\label{expansion-at-infinity}
In this section, we are going to establish the expansion of such inverted hypersurface near infinity. We first prove an approximation estimates for Willmore energy by the triharmonic energy.
\begin{lemma}\label{technique--lemma}
    Let $\Phi: \bB_R(0) \to \mathbb{R}^n$ be a four dimensional graph on $\bB_R(0)\subset \bR^4$, i.e. $\Phi=(x, \phi)$, where $x\in \bR^4$, $\phi\in \bR^{n-4}$ with induced metric $G=(g_{\alpha\beta})_{1 \leqslant \alpha,\beta \leqslant 4}$ and the averaged mean curvature vector $H$. If $|D \phi| \leqslant 1$ on $B_R$, then for a universal constant $C < \infty$, such that 
    \begin{equation}
       \left| |\nabla^\perp H|^2 \sqrt{\det G}-\frac{1}{16}|\nabla \Delta \phi|^2\right| \leqslant C \left(|D \phi|^2 |D^3 \phi|^2+|D^2\phi|^4\right).
    \end{equation}
\end{lemma}
\begin{proof}
Without loss of generality, we may assume that $\Phi$ satisfies $\phi(0)=0$, $D\phi(0)=0$. For clarity, we define
\begin{equation*}
\begin{split}
    g_{\top}=\Id_{4\times 4}+D^{T}\phi\cdot D\phi,\quad &g_{\bot}=\Id_{(n-4)\times (n-4)}+D\phi \cdot D^{T}\phi,\\
    g^{-1}_{\top}=\Id_{4\times 4}-D^{T}\phi\cdot g_{\bot}^{-1}\cdot D\phi, \quad &g_{\bot}^{-1}=\Id_{(n-4)\times (n-4)}-D\phi \cdot g^{-1}_{\top}\cdot D^{T}\phi.
    \end{split}
\end{equation*}
Note that $g_{\top}$ and $g^{-1}_{\top}$ are the induced metric $g$ and its inverse $g^{-1}$, respectively.

By QR factorization, there exist an $n\times (n-4)$ matrix $Q$ and an $(n-4)\times (n-4)$ matrix $R$ such that 
\begin{equation*}
    \begin{bmatrix}
        \Id \\
        D^{T}\phi
    \end{bmatrix}=Q\cdot R,
\end{equation*}
where $Q$ consists of orthonormal column vectors and $R$ is an upper-triangular matrix from Gram--Schmidt process. Moreover, we have 
\begin{equation*}
    g_{\bot}=\begin{bmatrix}
        \Id D\phi
    \end{bmatrix} \cdot
    \begin{bmatrix}
        \Id \\
        D^{T}\phi
    \end{bmatrix}=R^{T}\cdot Q^{T}\cdot Q\cdot R=R^{T}\cdot R,
\end{equation*}
and $\{w_{\alpha'}\}_{\alpha'=1}^{n-4}$ are orthonormal vectors with respect to $g_{\bot}$ where $\{w_{\alpha'}\}_{\alpha'=1}^{n-4}$ are column vectors of $R^{-1}$ which is denoted by $g_{\bot}^{-\frac{1}{2}}$ formally. Therefore, 
\begin{equation*}
    \boldsymbol{n}_{\alpha'}:=\left[-D^{T}\phi\cdot w_{\alpha'}, w_{\alpha'}\right], \quad \alpha'=1, \cdots, n-4
\end{equation*}
form an orthonormal basis of the normal bundle. By definition, we know that $w_{\alpha'}$ and $\boldsymbol{n}_{\alpha'}$ are only involving the first derivative of $\phi$. 

By the standard calculation, the averaged mean curvature vector are given by 
\begin{equation}\label{H on graph}
    H=\frac{1}{4} \tr_{G}\left(\pi_{n} D^2 \Phi\right),
\end{equation}
where $\pi_{n}=\Id_{\bR^n}-D \Phi G^{-1} D\Phi^{T}$ is the projection onto the normal bundle.

For any $ 1\leqslant \alpha \leqslant 4$ and $1 \leqslant \alpha', \beta' \leqslant n-4$,
\begin{equation}\label{perp nabla H}
    (\nabla_{\alpha}^\perp H)^{\beta'}=\nabla_{\alpha} H^{\beta'} -\Gamma_{\alpha\alpha'}^{\beta'} H^{\alpha'},
\end{equation}
where Christoffel symbol $\Gamma_{\alpha\alpha'}^{\beta'}$ is given by $\langle \nabla_{\alpha} n_{\alpha'}, n_{\beta'}\rangle$. From the definition of $Q$, we know that
\begin{equation}
    |\Gamma_{\alpha\alpha'}^{\beta'}|\leqslant C|D^2\phi||D\phi|.
\end{equation}

Besides, we can decompose $H=H^{\text{hor}}+H^{\text{ver}}\in \mathbb{R}^4 \oplus \mathbb{R}^{n-4}$ and obtain from \ref{H on graph} the equations
\begin{equation}\label{H hor}
   H^{\text{hor}}=-\frac{1}{4}g^{\alpha \beta}G^{-1}D \phi^{T} \partial_{\alpha \beta}^2 \phi,
\end{equation}
\begin{equation}\label{H ver}
    H^{\text{ver}}=\frac{1}{4}g^{\alpha \beta}(\Id- D \phi G^{-1} D \phi^{T})\cdot \partial_{\alpha \beta}^{2} \phi.
\end{equation}
On one hand, by \ref{H hor}, we know that
\begin{equation}
    |\nabla H^{\text{hor}}| \leqslant C(|D \phi||D^3 \phi|+|D^2 \phi|^2),
\end{equation}
On the other hand, by \ref{H ver}, we have 
\begin{align}
    |\nabla H^{\text{ver}}-\nabla \frac{1}{4}\Delta \phi| &\leqslant \left|(g^{\alpha\beta}-\delta_{\alpha\beta})\cdot\partial_{\alpha\beta} \nabla \phi\right|+C |D\phi| \cdot \left|D^2 \phi\right|^2\\
    &\leqslant C\left(|D \phi|^2 |D^3 \phi|+|D\phi||D^2 \phi|^2\right).
\end{align}
We also have the following identity
\begin{align}
  |\nabla^\perp H|^2 \sqrt{\det G}-\frac{1}{16}|\nabla \Delta\phi|^2&=|\nabla^\perp H^{\text{ver}}-\frac{1}{4}\nabla \Delta \phi|^2+\frac{1}{2}\langle\nabla\Delta \phi, \nabla^\perp H^{\text{ver}}-\frac{1}{4}\nabla\Delta\phi\rangle\\
  &+|\nabla^\perp H^{\text{ver}}|^2 (\sqrt{\det{G}}-1)+|\nabla^\perp H^{\text{hor}}|^2 \sqrt{\det{G}}.
\end{align}
Hence, with the fact that $\sqrt{\det G}\leqslant 1+|D\phi|^2$, we have 
\begin{equation*}
\begin{split}
  \left||\nabla^\perp H|^2 \sqrt{\det G}-\frac{1}{16}|\nabla \Delta\phi|^2\right| &\leqslant C(|D\phi|^2|D^3\phi|+|D\phi||D^2\phi|^2)^2 \\
  &+C(|D^3\phi|^2|D\phi|^2+|D^3\phi||D\phi||D^2\phi|^2)\\
  &+C|D\phi|^2(|D^3\phi|+|D\phi|^2|D^3\phi|+|D\phi||D^2\phi|^2)^2\\
  &+C(|D\phi||D^3\phi|+|D^2\phi|^2)^2\\
  &\leqslant C(|D^3\phi|^2|D\phi|^2 +|D^2\phi|^4).
  \end{split}
\end{equation*}
Here, we used the assumption $|D\phi| \leqslant 1$, which yields desired result.
\end{proof}
The following lemma gives the precise description of the inverted hypersurface near infinity.
\begin{lemma}
    Let $\Phi: \mathbb{B}_{r}(0)\to \bR^n$, $\Phi(z)=(z,\phi(z))$, be a four-dimensional graph where $\phi\in C^7(\bB_r(0), \bR^{n-4})$ satisfies
    \begin{equation*}
        \phi(0)=0, \quad D\phi(0)=0.
    \end{equation*}
    Denote by $P=D^2\phi(0):\bR^4\times \bR^4\to \bR^{n-4}$ the second fundamental form of $\Phi$ at the origin. After possibly restricting to a smaller ball again called $\mathbb{B}_r(0)$, the inverted submanifold
    \begin{equation*}
        \hat{\Phi}=I\circ \Phi: \bB_r(0)\backslash \{0\}\to \bR^n, \quad \hat{\Phi}=\frac{\Phi}{|\Phi|^2},
    \end{equation*}
    has a graph representation over some open neighborhood $\hat{U}\subset \bR^4$ of infinity. The corresponding graph function $\hat{\phi}:\hat{U}\to \bR^{n-4}$ has an expansion
    \begin{equation}
        \hat{\phi}(\zeta)=\frac{1}{2}P\left(\frac{\zeta}{|\zeta|},\frac{\zeta}{|\zeta|} \right)+\frac{1}{6}Q\left(\frac{\zeta}{|\zeta|},\frac{\zeta}{|\zeta|},\frac{\zeta}{|\zeta|} \right)|\zeta|^{-1}+\hat{\varphi}(\zeta)
    \end{equation}
where
\begin{equation*}
    \sum_{i=0}^{3} |\zeta|^{i+2}|D^{i}\varphi(\zeta)|\leq C, \quad \forall \zeta\in \hat{U}.
\end{equation*}
\end{lemma}

\begin{proof}
    Since the proof is totally similar to that in two-dimensional case, we only sketch it here. We can define diffeomorphisms $\Phi: \mathbb{B}_r(0)\to U$ and $\Psi:U\to \mathbb{B}_r(0)$ by
    \begin{equation*}
        \Phi:=\left(I\circ \pi\circ I\right)(z,\phi(z)),\quad \Psi=\Phi^{-1}, 
    \end{equation*}
    where $U$ is an open neighborhood of the origin and $\pi$ is the projection onto the first factor so that for all $ w\in U\backslash \{0\}$
    \begin{equation*}
        |w|^{-3}|\Psi(w)-w|+|w|^{-2}|D\Psi(w)-\Id|+|w|^{-1}\left(|D^2\Psi(w)|+{|D^3\Psi(w)|}\right)\leq C.
    \end{equation*}
%The above decay follows from the decay of the function $\eta(z)=\frac{|u(z)|^2}{|z|^2}$ near the origin. Since we can only assume that 
 %\begin{equation*}
  %      u(0)=0, \quad Du(0)=0.
   % \end{equation*}
% which implies that $|D^2\eta|\leq C$ and trivially $|D^3 \eta|\leq C$.
Hence, if we define $\psi: \hat{U}=I(U)\to \bR^4$, $\psi(\zeta)=\Psi(I(\zeta))-I(\zeta)$, we have
\begin{equation*}
    |\zeta|^3|\psi(\zeta)|+|\zeta|^4|D\psi(\zeta)|+|\zeta|^5|D^2\psi(\zeta)|+|\zeta|^6|D^3\psi(\zeta)|\leq C, \quad \forall\zeta \in \hat{U}.
\end{equation*}
(Here we replace $|\zeta|^7|D^3\psi(\zeta)|$ by $|\zeta|^6|D^3\psi(\zeta)|$ because $|\zeta|$ is sufficiently large.)
Therefore, $\hat{u}$ is defined by
\begin{equation*}
    \hat{\phi}=\frac{|\zeta|^2u(\Psi(I(\zeta)))}{1+\langle \psi(\zeta), \zeta\rangle}, \quad \forall \zeta\in \hat{U}.
\end{equation*}
The denominator satisfies
\begin{equation*}
    |\zeta|^2|\mu(\zeta)|+|\zeta|^3|D\mu(\zeta)|+|\zeta|^4|D^2\mu(\zeta)|+|\zeta|^5|D^3\mu(\zeta)|\leq C, \quad \forall\zeta \in \hat{U},
\end{equation*}
where $\mu(\zeta)=\left(1+\langle \psi(\zeta), \zeta\rangle\right)^{-1}-1$.
By our assumptions, we have
\begin{equation*}
    \phi(z)=\frac{1}{2}P\left(z,z\right)+\frac{1}{6}Q(z,z,z)+\varphi(z),
\end{equation*}
where
\begin{equation*}
    |z|^{-4}|\varphi(z)|+|z|^{-3}|D\varphi(z)|+|z|^{-2}|D^2\varphi(z)|+|z|^{-1}|D^3\varphi(z)|\leq C, \quad \forall z\in \mathbb{B}_r(0)\backslash \{0\}.
\end{equation*}
It follows that the nominator satisfies
\begin{equation*}
    |\zeta|^2\phi(\Psi(I(\zeta)))=\frac{1}{2}P\left(\frac{\zeta}{|\zeta|},\frac{\zeta}{|\zeta|} \right)+\frac{1}{6}Q\left(\frac{\zeta}{|\zeta|},\frac{\zeta}{|\zeta|},\frac{\zeta}{|\zeta|} \right)|\zeta|^{-1}+|\zeta|^2\lambda(\zeta),
\end{equation*}
where
\begin{equation*}
\begin{split}
    \lambda(\zeta)&=\frac{1}{2}\left\{2P(I(\zeta), \psi(\zeta))+P(\psi(\zeta), \psi(\zeta))\right\}\\
    &+\frac{1}{6}\left\{3Q(I(\zeta), \psi(\zeta), \psi(\zeta))+3Q(I(\zeta), I(\zeta), \psi(\zeta))+Q(\psi(\zeta), \psi(\zeta), \psi(\zeta))\right\}\\
    &+\varphi(\Psi(I(\zeta))).
    \end{split}
\end{equation*}
From above discussion, we conclude that
%\begin{equation*}
 %   \begin{split}
  %      &|2P(I(\zeta), \psi(\zeta))+P(\psi(\zeta), \psi(\zeta))|\leq C|\zeta|^{-4},\\
   %     &|3Q(I(\zeta), \psi(\zeta), \psi(\zeta))+3Q(I(\zeta), I(\zeta), \psi(\zeta))+Q(\psi(\zeta), \psi(\zeta), \psi(\zeta))|\leq C|\zeta|^{-5},\\
    %    &|\varphi(\Psi(I(\zeta)))|\leq C|\zeta|^{-4},
    %\end{split}
%\end{equation*}
%which implies that
\begin{equation*}
    \sum_{i=0}^{3} |\zeta|^{i+2}|D^{i}\lambda(\zeta)|\leq C|\zeta|^{-2}.
\end{equation*}
The error term $\hat{\varphi}$ is given by
\begin{equation*}
    \begin{split}
        \hat{\varphi}(\zeta)&=\left(\frac{1}{2}P\left(\frac{\zeta}{|\zeta|},\frac{\zeta}{|\zeta|} \right)+\frac{1}{6}Q\left(\frac{\zeta}{|\zeta|},\frac{\zeta}{|\zeta|},\frac{\zeta}{|\zeta|} \right)|\zeta|^{-1}\right)\mu(\zeta)\\
        &+|\zeta|^2\lambda(\zeta)\left(1+\langle \psi(\zeta), \zeta\rangle\right)^{-1}
    \end{split}
\end{equation*}
which satisfies the desired estimate.
\end{proof}
\section{Construction of the connected sum}\label{connected-sum}
Let $\Phi_i:\Sigma_i\to \bR^n$, $i=1,2$, be $4-$dimensional closed immersed submanifolds of class $C^7$, and suppose that
\begin{equation*}
    \Phi^{-1}_i\{0\}=\{p_i\}, ~~\mbox{for some}~~ p_i\in \Sigma_i, \quad \im D\Phi_i (p_i)=\bR^4\times \{0\}, ~~\mbox{for}~~ i=1,2.
\end{equation*}
Without lose of generality, we assume that $\Phi_1$ is not the immersion of round sphere. 
Note that for any pair of multiplicity one points $X_i\in \Phi_i (\Sigma_i), i=1,2$, the above situation can be arranged by appropriate translations and rotations. For some $r>0$, there are local graph representations $(z, u(z))$ for $\Phi_1$ and $(z, v(z))$ for $\Phi_2$ of the form on $\mathbb{B}_{r}(0)$
\begin{equation}
        u(z)=p(z)+\varphi(z), \quad v(z)=q(z)+\psi(z),
\end{equation}
where $p(z)$ and $q(z)$ are polynomials of the form
\begin{equation}
    p(z)=\frac{1}{2}P(z,z)+\frac{1}{6}Q(z,z,z), \quad q(z)=\frac{1}{2}R(z,z)+\frac{1}{6}S(z,z,z).
\end{equation}
Here $P,R: \bR^4\times \bR^4\to \bR^{n-4}$ denote the associated second fundamental forms at the origin for $\Phi_1$ and $\Phi_2$, respectively. The error terms satisfy for any $z\in \mathbb{B}_r(0)\backslash \{0\}$
\begin{equation}
    \sum_{i=0}^{3} |z|^{-4+i}|D^{i}\varphi|\leq C, \quad \sum_{i=0}^{3} |z|^{-4+i}|D^{i}\psi|\leq C.
\end{equation}
We decompose the quadratic form $P$ in its pure trace part and its tracefree part,
\begin{equation*}
     P(z,z)=\frac{1}{4}(\tr P)|z|^2+\mathring{P}(z,z),
\end{equation*}
and consider the immersion $\mathring{\Phi}_1: \Sigma_1\backslash \{p_1\}\to \bR^n$, $\mathring{\Phi}_1(p)=\hat{\Phi}_1(p)-\frac{1}{4}\tr P$.
Using the expansion at infinity, we infer that for some $R<\infty$ the submanifold $\mathring{\Phi}_1$ has a graph representation $(\zeta, \mathring{u}(\zeta))$ at infinity of the form
\begin{equation*}
    \begin{split}
        &\mathring{u}(\zeta)=\mathring{p}(\zeta)+\mathring{\varphi}(\zeta) ~~\mbox{on}~~ \bR^4\backslash \mathbb{B}_R(0), \\
        & \sum_{i=0}^{3} |\zeta|^{i+2}|D^{i}\mathring{\varphi}(\zeta)|\leq C, \quad \forall \zeta\in \bR^4\backslash \mathbb{B}_R(0).
    \end{split}
\end{equation*}
where
\begin{equation}\label{expression-of-traceless-p}
    \mathring{p}(\zeta)=\frac{1}{2}\mathring{P}\left(\frac{\zeta}{|\zeta|},\frac{\zeta}{|\zeta|} \right)+\frac{1}{6}Q\left(\frac{\zeta}{|\zeta|},\frac{\zeta}{|\zeta|},\frac{\zeta}{|\zeta|} \right)|\zeta|^{-1}.
\end{equation}
 For the construction of the pasted submanifold we will blow down $\mathring{\Phi}_1$ and blow up $\Phi_2$. In general, the scaling of the graph of a function $w:\Omega\subset \bR^4 \to \bR^{n-4}$ with a factor $\lambda>0$ corresponds to the scaling of the function $w$ given by
\begin{equation*}
    w_{\lambda}: \Omega_{\lambda}=\{z\in\bR^4: \lambda^{-1}z\in \Omega\}\to \bR^{n-4}, \quad w_{\lambda}(z)=\lambda w(\lambda^{-1} z).
\end{equation*}
Now for $\alpha, \beta>0$ small enough, consider the scaled submanifolds $ \mathring{\Phi}_{1,\alpha}$ and $\Phi_{2,\beta^{-1}}$ which have graph representations
\begin{equation}
    \begin{split}
        &\mathring{u}_{\alpha}(z)=\mathring{p}_{\alpha}(z)+\mathring{\varphi}_{\alpha}(z) ~~\mbox{on}~~ \bR^4\backslash \mathbb{B}_{\alpha R}(0),\\
        & v_{\beta^{-1}}(z)=q_{\beta^{-1}}(z)+\psi_{\beta^{-1}}(z) ~~\mbox{on}~~ \mathbb{B}_{r\beta^{-1}}(0).
    \end{split}
\end{equation}
From above discussion, we obtain the bounds on error terms
\begin{equation}\label{rescale-decay}
\begin{split}
    &\sum_{i=0}^{3} |z|^{i+2}|D^{i}\mathring{\varphi}_{\alpha}(z)|\leq C\alpha^3  ~~\mbox{on}~~ \bR^4\backslash \bB_{\alpha R}(0),\\
    &\sum_{i=0}^{3} |z|^{-4+i}|D^{i}\psi_{\beta^{-1}}(z)|\leq C\beta^3  ~~\mbox{on}~~ \bB_{r\beta^{-1}}(0)\backslash \{0\}.
    \end{split}
\end{equation}

Choose a third small parameter $\gamma$ where $0<\alpha, \beta \ll \gamma$. Essentially, the construction of the pasted submanifold then is as follows: from the submanifold $ \mathring{\Phi}_{1,\alpha}$ we remove the end given by the graph $\mathring{u}_{\alpha}$ over $\bR^4\backslash \bB_{\gamma}(0)$, and from the submanifold $\Phi_{2,\beta^{-1}}$ we remove the disk given by the graph of $v_{\beta^{-1}}$ on $\bB_1(0)$. Then we connect the two submanifolds by the graph of a function $w: \bB_1(0)\backslash \bB_{\gamma}(0)\to \bR^{n-4}$, which suitably interpolates between $\mathring{u}_{\alpha}$ and $v_{\beta^{-1}}$.

Actually, a slight modification of this is more convenient. For given $\alpha>0$, we fix a function $\eta\in C^{\infty}(\bR)$ with the properties
\begin{equation*}
\begin{split}
    &\eta(s)=\left\{ 
    \begin{array}{cc}
      0  & ~~\mbox{for}~~ s\leq \frac{\sqrt{\alpha}}{4}, \\
      1  & ~~\mbox{for}~~ s\geq \frac{3\sqrt{\alpha}}{4},
    \end{array}
    \right.\\
    & |\eta|+\alpha^{\frac{1}{2}}|\eta'|+\alpha |\eta''|+\alpha^{\frac{3}{2}}|\eta'''|\leq C ~~\mbox{with C independent of}~~ \alpha.  
    \end{split}
\end{equation*}
We cut off the higher-order terms in the expansions of $\mathring{u}_{\alpha}$ and $v_{\beta^{-1}}$ by defining, for $r=|z|$,
\begin{equation*}
    w(z)=\left\{
    \begin{array}{cc}
       \mathring{p}_{\alpha}(z)+\eta(\gamma-r)\mathring{\varphi}_{\alpha}(z) &~~\mbox{for}~~  \alpha R<r< \gamma,  \\
       q_{\beta^{-1}}(z)+\eta(r-1)\psi_{\beta^{-1}}(z)   & ~~\mbox{for}~~ 1\leq r\leq \beta^{-1}r.
    \end{array}
    \right.
\end{equation*}
On the remaining annulus $\bB_1(0)\backslash \bB_{\gamma}(0)$, we define the interpolation $w$ as the unique solution of the equation
\begin{equation}\label{trilaplacian}
    \Delta_0^3 w=0 \quad ~~\mbox{on}~~ \bB_1(0)\backslash \bB_{\gamma}(0)
\end{equation}
subject to boundary conditions
\begin{equation}\label{boundary-conditions}
    \begin{split}
        & w=\mathring{p}_{\alpha}, \quad \partial_r w=\partial_r \mathring{p}_{\alpha}, \quad \Delta_0 w=\Delta_0 \mathring{p}_{\alpha}\quad ~~\mbox{on}~~ |z|=\gamma,\\
        & w=q_{\beta^{-1}}, \quad \partial_r w=\partial_r q_{\beta^{-1}}, \quad \Delta_0 w=\Delta_0 q_{\beta^{-1}}\quad ~~\mbox{on}~~ |z|=1,
    \end{split}
\end{equation}
where $\partial_r=\frac{\partial}{\partial r}$ and $\Delta_0=\sum_{i=1}^{4}\frac{\partial^2}{\partial x_i^2}$. The main purpose of this section is devoted to the explicit computation of the solution to the problem (\ref{trilaplacian}) and (\ref{boundary-conditions}). 

\subsection{Formal solution}\label{formal-solution}
First, we find formal solutions of (\ref{trilaplacian}).
\begin{lemma}\label{e}
    If $e(r)$ is a radial function satisfying $\Delta_0^3 (e(r))=0$, $e(r)$ must be in the form of
    \begin{equation}
        e(r)=\sum_{i=1}^{6}A_i e_i(r), \quad A_i\in \bR^{n-4},
    \end{equation}
where
\begin{equation*}
    \begin{split}
        &e_1(r)=r^{-2}, \quad e_2(r)=1, \quad e_3(r)=\ln r,\\
        &e_4(r)=r^2, \quad e_5(r)=r^2\ln r, \quad e_6(r)=r^4.
    \end{split}
\end{equation*}
If we define $E(r)=[e_1(r), e_2(r), e_3(r), e_4(r), e_5(r), e_6(r)]$, the matrix $\mM^{e}$ encodes values of $\{e_i(r)\}_{i=1}^{6}$ up to the second derivative on $r=\gamma$ and $r=1$:
\begin{equation}\label{Me}
\mM^{e}:=\begin{bmatrix}
    E(\gamma) &\\
    \partial_r E(\gamma) &\\
    \partial^2_r E(\gamma) &\\
    E(1) &\\
    \partial_r E(1) &\\
    \partial^2_r E(1) &\\
\end{bmatrix}=\begin{bmatrix}
\gamma^{-2} & 1 & \ln \gamma & \gamma^2 & \gamma^2\ln \gamma& \gamma^4\\
-2\gamma^{-3} & 0 & \gamma^{-1} & 2\gamma & 2\gamma\ln \gamma+\gamma& 4\gamma^3\\
6\gamma^{-4} & 0 & - \gamma^{-2} & 2 & 2\ln \gamma+3& 12\gamma^2\\
1 & 1 & 0 & 1 & 0& 1\\
-2 & 0 & 1 & 2 & 1& 4\\
6 & 0 & -1 & 2 & 3& 12
\end{bmatrix}.
\end{equation}
Moreover, the Laplace operator $\Delta_0$ acts on the six-dimensional space $E(r)$ as the matrix:
\begin{equation}\label{eLaplace}
    \Delta_0(E(r))=E(r)\cdot 
    \begin{bmatrix}
0 & 0 & 2 & 0 & 0& 0\\
0 & 0 & 0 & 8 & 6& 0\\
0 & 0 & 0 & 0 & 8& 0\\
0 & 0 & 0 & 0 & 0& 24\\
0 & 0 & 0 & 0 & 0& 0\\
0 & 0 & 0 & 0 & 0& 0
\end{bmatrix}.
\end{equation}
\end{lemma}
\begin{proof}
    Note that
    \begin{equation*}
        \Delta_0=\frac{\partial^2}{\partial r^2}+\frac{3}{r}\cdot\frac{\partial}{\partial r}+\frac{1}{r^2}\cdot \Delta_{S^3}.
    \end{equation*}
Acting on $r^k$, we have
\begin{equation*}
    \Delta_0 (r^k)=k(k+2)r^{k-2}.
\end{equation*}
Therefore, from
\begin{equation*}
    0=\Delta^3_0 (r^k)=(k+2)\cdot k^2 \cdot (k-2)^2\cdot (k-4)r^{k-6},
\end{equation*}
we know that
\begin{equation*}
    \begin{split}
        k&=-2, \quad e_1(r)=r^{-2},\\
        k&=0, \quad e_2(r)=1, \quad e_3(r)=\ln r,\\
        k&=2, \quad e_4(r)=r^2, \quad e_5(r)=r^2\ln r,\\
        k&=4, \quad e_6(r)=r^4.
    \end{split}
\end{equation*}
\end{proof}

\begin{lemma}\label{f}
    If $f(r)$ is a radial function satisfying $\Delta_0^3 (f(r)X_i)=0$, for some $X_i\in \mathcal{H}^{\bS}_{1}$, $f(r)$ must be in the form of
    \begin{equation}
        f(r)=\sum_{i=1}^{6}B_i F_i(r), \quad B_i\in \bR^{n-4},
    \end{equation}
where
\begin{equation*}
    \begin{split}
        &F_1(r)=r^{-3}, \quad F_2(r)=r^{-1}, \quad F_3(r)=r,\\
        &F_4(r)=r\ln r, \quad F_5(r)=r^3, \quad F_6(r)=r^5.
    \end{split}
\end{equation*}
Moreover, the matrix $\mM^{f}$ encodes values of $\{F_i(r)\}_{i=1}^{6}$ on $r=\gamma$ and $r=1$:
\begin{equation}\label{Mf}
\mM^{f}:=\begin{bmatrix}
\gamma^{-3} & \gamma^{-1} &  \gamma & \gamma \ln \gamma & \gamma^3& \gamma^5\\
-3\gamma^{-4} & -\gamma^{-2} & 1 & 1+\ln \gamma & 3\gamma^2& 5\gamma^4\\
12\gamma^{-5} & 2\gamma^{-3} & 0 & \gamma^{-1} & 6 \gamma& 20\gamma^3\\
1 & 1 & 1 & 0 & 1& 1\\
-3 & -1 & 1 & 1 & 3& 5\\
12 & 2 & 0 & 1 & 6& 20
\end{bmatrix}.
\end{equation}

\end{lemma}
\begin{proof}
    Note that
    \begin{equation*}
        \Delta_{S^3}(X_i)=-3X_i.
    \end{equation*}
Acting on $r^kX_i$, we have
\begin{equation*}
    \Delta_0 (r^kX_i)=(k-1)(k+3)r^{k-2}X_i.
\end{equation*}
Therefore, from
\begin{equation*}
    0=\Delta^3_0 (r^kX_i)=(k-5)\cdot(k-3)\cdot(k-1)^2\cdot(k+1)\cdot(k+3)r^{k-6}X_i,
\end{equation*}
we know that
\begin{equation*}
    \begin{split}
        k&=-3, \quad F_1(r)=r^{-3},\\
        k&=-1, \quad F_2(r)=r^{-1},\\
        k&=1, \quad F_3(r)=r, \quad F_4(r)=r\ln r,\\
        k&=3, \quad F_5(r)=r^3,\\
        k&=5, \quad F_6(r)=r^5.
    \end{split}
\end{equation*}
\end{proof}

\begin{lemma}\label{g}
    If $g(r)$ is a radial function satisfying $\Delta_0^3 (g(r)Y_i)=0$, for some $Y_i\in \mathcal{H}^{\bS}_{2}$, $g(r)$ must be in the form of
    \begin{equation}
        g(r)=\sum_{i=1}^{6}C_i G_i(r), \quad C_i\in \bR^{n-4},
    \end{equation}
where
\begin{equation*}
    \begin{split}
        &G_1(r)=r^{-4}, \quad G_2(r)=r^{-2}, \quad G_3(r)=1,\\
        &G_4(r)=r^2, \quad G_5(r)=r^4, \quad G_6(r)=r^6.
    \end{split}
\end{equation*}
Moreover, the matrix $\mM^{g}$ encodes values of $\{G_i(r)\}_{i=1}^{6}$ on $r=\gamma$ and $r=1$:
\begin{equation}\label{Mg}
\mM^{g}:=\begin{bmatrix}
\gamma^{-4} & \gamma^{-2} &  1 & \gamma^2 & \gamma^4& \gamma^6\\
-4\gamma^{-5} & -2\gamma^{-3} & 0 & 2 \gamma & 4\gamma^3& 6\gamma^5\\
20\gamma^{-6} & 6\gamma^{-4} & 0 & 2 & 12 \gamma^2& 30\gamma^4\\
1 & 1 & 1 & 1 & 1& 1\\
-4 & -2 & 0 & 2 & 4& 6\\
20 & 6 & 0 & 2 & 12& 30
\end{bmatrix}.
\end{equation}

\end{lemma}
\begin{proof}
    Note that
    \begin{equation*}
        \Delta_{S^3}(Y_i)=-8Y_i.
    \end{equation*}
Acting on $r^kY_i$, we have
\begin{equation*}
    \Delta_0 (r^kY_i)=(k-2)(k+4)r^{k-2}Y_i.
\end{equation*}
Therefore, from
\begin{equation*}
    0=\Delta^3_0 (r^kY_i)=(k-6)\cdot (k-4)\cdot (k-2)\cdot k\cdot(k+2)\cdot(k+4)r^{k-6}Y_i,
\end{equation*}
we know that
\begin{equation*}
    \begin{split}
        k&=-4, \quad G_1(r)=r^{-4},\\
        k&=-2, \quad G_2(r)=r^{-2},\\
        k&=0, \quad G_3(r)=1,\\
        k&=2, \quad G_4(r)=r^2,\\
        k&=4, \quad G_5(r)=r^4,\\
        k&=6, \quad G_6(r)=r^6.
    \end{split}
\end{equation*}
\end{proof}

\begin{lemma}\label{h}
    If $h(r)$ is a radial function satisfying $\Delta_0^3 (h(r)W_i)=0$, for some $W_i\in \mathcal{H}^{\bS}_{3}$, $h(r)$ must be in the form of
    \begin{equation}
        h(r)=\sum_{i=1}^{6}D_i H_i(r), \quad D_i\in \bR^{n-4},
    \end{equation}
where
\begin{equation*}
    \begin{split}
        &H_1(r)=r^{-5}, \quad H_2(r)=r^{-3}, \quad H_3(r)=r^{-1},\\
        &H_4(r)=r^3, \quad H_5(r)=r^5, \quad H_6(r)=r^7.
    \end{split}
\end{equation*}
Moreover, the matrix $\mM^{h}$ encodes values of $\{H_i(r)\}_{i=1}^{6}$ on $r=\gamma$ and $r=1$:
\begin{equation}\label{Mh}
\mM^{h}:=\begin{bmatrix}
\gamma^{-5} & \gamma^{-3} &  \gamma^{-1} & \gamma^3 & \gamma^5& \gamma^7\\
-5\gamma^{-6} & -3\gamma^{-4} & -\gamma^{-2} & 3\gamma^2 & 5\gamma^4& 7\gamma^6\\
30\gamma^{-7} & 12\gamma^{-5} & 2\gamma^{-3} & 6\gamma & 20 \gamma^3& 42\gamma^5\\
1 & 1 & 1 & 1 & 1& 1\\
-5 & -3 & -1 & 3 & 5& 7\\
30 & 12 & 2 & 6 & 20& 42
\end{bmatrix}.
\end{equation}

\end{lemma}
\begin{proof}
    Note that
    \begin{equation*}
        \Delta_{S^3}(W_i)=-15W_i.
    \end{equation*}
Acting on $r^kW_i$, we have
\begin{equation*}
    \Delta_0 (r^kW_i)=(k-3)(k+5)r^{k-2}W_i.
\end{equation*}
Therefore, from
\begin{equation*}
    0=\Delta^3_0 (r^kW_i)=(k-7)\cdot (k-5)\cdot (k-3)\cdot (k+1)\cdot(k+3)\cdot(k+5)r^{k-6}W_i,
\end{equation*}
we know that
\begin{equation*}
    \begin{split}
        k&=-5, \quad H_1(r)=r^{-5},\\
        k&=-3, \quad H_2(r)=r^{-3},\\
        k&=-1, \quad H_3(r)=r^{-1},\\
        k&=3, \quad H_4(r)=r^3,\\
        k&=5, \quad H_5(r)=r^5,\\
        k&=7, \quad H_6(r)=r^7.
    \end{split}
\end{equation*}
\end{proof}

Combining Lemma \ref{e}, \ref{f}, \ref{g} and \ref{h}, we may assume that $w$ is in the form of
\begin{equation*}
    w=e(r)+\sum_{i=1}^{N_1}f_i(r)X_i+\sum_{i=1}^{N_2}g_i(r)Y_i+\sum_{i=1}^{N_3}h_i(r)W_i,
\end{equation*}
where
\begin{equation}\label{w}
    \begin{split}
        &e(r)=\sum_{l=1}^{6}A_l e_l(r), \quad f_i(r)=\sum_{l=1}^{6}B^l_i F_l(r),\\
        &g_i(r)=\sum_{l=1}^{6}C^l_i G_l(r),\quad h_i(r)=\sum_{l=1}^{6}D^l_i H_l(r).
    \end{split}
\end{equation}

\subsection{Solving coefficients}\label{coefficients}
Next, to determine coefficients in the expansion (\ref{w}), we evaluate boundary values of $w$. 

Before that, we write down the expansion of (\ref{boundary-conditions}) in terms of spherical harmonics.
\begin{equation}
    \begin{split}
        w&=\frac{\alpha}{2}\mathring{P}\left(\frac{z}{|z|},\frac{z}{|z|} \right)+\frac{\alpha^2}{6}Q\left(\frac{z}{|z|},\frac{z}{|z|},\frac{z}{|z|} \right)\gamma^{-1},\\
        \partial_r w&=-\frac{\alpha^2}{6}Q\left(\frac{z}{|z|},\frac{z}{|z|},\frac{z}{|z|} \right)\gamma^{-2},\\
        \partial^2_r w&=\frac{\alpha^3}{3}Q\left(\frac{z}{|z|},\frac{z}{|z|},\frac{z}{|z|} \right)\gamma^{-3}, \quad ~~\mbox{for}~~ |z|=\gamma,
    \end{split}
\end{equation}
and 
\begin{equation}
    \begin{split}
        w&=\frac{\beta}{2}R\left(z,z \right)+\frac{\beta^2}{6}S\left(z,z,z \right),\\
        \partial_r w&=\beta R\left(z,z \right)+\frac{\beta^2}{2}S\left(z,z,z \right),\\
        \partial^2_r w&=\beta R\left(z,z \right)+\beta^2 S\left(z,z,z \right), \quad ~~\mbox{for}~~ |z|=1.
    \end{split}
\end{equation}
Moreover, by (\ref{expression-of-traceless-p}), we may expand $\mathring{P}, Q, R, S$ as follows:
\begin{equation}\label{PQRS-expansion}
    \begin{split}
        \mathring{P}\left(\frac{z}{|z|},\frac{z}{|z|} \right)=\sum_{i=1}^{N_2}p_i\cdot Y_i, &\quad Q\left(\frac{z}{|z|},\frac{z}{|z|},\frac{z}{|z|} \right)=\sum_{i=1}^{N_1} q^{(1)}_i\cdot X_i+\sum_{i=1}^{N_3}q^{(3)}_i \cdot W_i;\\
        R\left(\frac{z}{|z|},\frac{z}{|z|} \right)=r_0+\sum_{i=1}^{N_2}r_i\cdot Y_i, &\quad S\left(\frac{z}{|z|},\frac{z}{|z|},\frac{z}{|z|} \right)=\sum_{i=1}^{N_1} s^{(1)}_i\cdot X_i+\sum_{i=1}^{N_3}s^{(3)}_i\cdot W_i.
    \end{split}
\end{equation}
Therefore, on $|z|=\gamma$, the boundary condition is
\begin{equation}
    \begin{split}
        w&=\frac{\alpha}{2}\sum_{i=1}^{N_2}p_i\cdot Y_i+\frac{\alpha^2}{6\gamma}\left(\sum_{i=1}^{N_1} q^{(1)}_i\cdot X_i+\sum_{i=1}^{N_3}q^{(3)}_i \cdot W_i\right),\\
        \partial_r w&=-\frac{\alpha^2}{6\gamma^2}\left(\sum_{i=1}^{N_1} q^{(1)}_i\cdot X_i+\sum_{i=1}^{N_3}q^{(3)}_i \cdot W_i\right),\\
        \partial^2_r w&=\frac{\alpha^2}{3\gamma^3}\left(\sum_{i=1}^{N_1} q^{(1)}_i\cdot X_i+\sum_{i=1}^{N_3}q^{(3)}_i \cdot W_i\right),
    \end{split}
\end{equation}
and on $|z|=1$, the boundary condition is
\begin{equation}
    \begin{split}
        w&=\frac{\beta}{2}\left(r_0+\sum_{i=1}^{N_2}r_i\cdot Y_i \right)+\frac{\beta^2}{6}\left(\sum_{i=1}^{N_1} s^{(1)}_i\cdot X_i+\sum_{i=1}^{N_3}s^{(3)}_i\cdot W_i \right),\\
        \partial_r w&=\beta \left(r_0+\sum_{i=1}^{N_2}r_i\cdot Y_i \right)+\frac{\beta^2}{2}\left(\sum_{i=1}^{N_1} s^{(1)}_i\cdot X_i+\sum_{i=1}^{N_3}s^{(3)}_i\cdot W_i \right),\\
        \partial^2_r w&=\beta \left(r_0+\sum_{i=1}^{N_2}r_i\cdot Y_i \right)+\beta^2 \left(\sum_{i=1}^{N_1} s^{(1)}_i\cdot X_i+\sum_{i=1}^{N_3}s^{(3)}_i\cdot W_i \right).
    \end{split}
\end{equation}
Note that in the above expansion, coefficients are vector valued. In the sequel, for simplicity, we will abuse notation and identify a vector with its components. As usual, we denote by $O(\alpha^i \beta^j\gamma^k (\ln \gamma)^l)$ a quantity that is bounded in absolute value by a constant times $\alpha^i \beta^j\gamma^k (\ln \gamma)^l$. Eventually, we will let $\beta=t\alpha<<\gamma<<1$ in Section \ref{energy}.

\begin{lemma}\label{corfficient-e}
    Coefficients $\{A_l\}_{l=1}^6$ satisfy
    \begin{equation*}
    \begin{bmatrix}
        A_1\\
        A_2\\
        A_3\\
        A_4\\
        A_5\\
        A_6
    \end{bmatrix}=\beta r_0\begin{bmatrix}
       -\frac{\gamma^4}{2}+O(\gamma^6\ln \gamma)\\
        2\gamma^2\ln \gamma+O(\gamma^4(\ln \gamma)^3)\\
        -2\gamma^2\left(1+16\gamma^2(\ln \gamma)^2\right)+O(\gamma^4\ln \gamma)\\
       \frac{1}{2}\left(1+4\gamma^2\right)+O(\gamma^4(\ln \gamma)^4)\\
        2\gamma^2\left(4\ln \gamma+3\right)+O(\gamma^4(\ln \gamma)^3)\\
        -\gamma^2 \left(\ln \gamma+1\right)+O(\gamma^4(\ln \gamma)^3)
    \end{bmatrix}.
\end{equation*}
\end{lemma}
\begin{proof}
Comparing constant terms, we know that coefficients $\{A_l\}_{l=1}^6$ satisfy the equation
\begin{equation}\label{eq:A}
    \mM^{e}\cdot \begin{bmatrix}
        A_1\\
        A_2\\
        A_3\\
        A_4\\
        A_5\\
        A_6
    \end{bmatrix}=\begin{bmatrix}
        0\\
        0\\
        0\\
        \frac{\beta}{2}r_0\\
        \beta r_0\\
        \beta r_0
    \end{bmatrix}.
\end{equation}
The determinant of $M^e$ and corresponding cofactors are given by
\begin{equation}
    \begin{split}
        \det \mM^e=&-16\gamma^5 + 256\gamma^3 \ln^2(\gamma) - 384\gamma^3 \ln(\gamma) + 272\gamma^3 + 384\gamma \ln(\gamma)\\ 
        &- 736\gamma+ \frac{384 \ln(\gamma)}{\gamma} + \frac{736}{\gamma} - \frac{256 \ln^2(\gamma)}{\gamma^3} - \frac{384 \ln(\gamma)}{\gamma^3} - \frac{272}{\gamma^3} + \frac{16}{\gamma^5},
    \end{split}
\end{equation}

\begin{equation}
\begin{split}
    \det \mM^e_1=&-8 \beta r_0\frac{ \gamma^6 - 4 \gamma^4 \ln(\gamma) - \gamma^4 + 4 \gamma^2 \ln(\gamma) - \gamma^2+ 1}{\gamma},
\end{split}    
\end{equation}

\begin{equation}
    \begin{split}
        \det \mM^e_2=&32 \beta r_0\frac{-4 \gamma^6 \ln(\gamma) + 3 \gamma^6 + 3 \gamma^4 \ln(\gamma) - 6 \gamma^4  + 3 \gamma^2 +  \ln(\gamma)}{\gamma^3},
    \end{split}
\end{equation}

\begin{equation}
    \begin{split}
        \det \mM^e_3=&32 \beta r_0\frac{- \gamma^8 + 4 \gamma^6 - 6 \gamma^4  + 4 \gamma^2  - 1}{\gamma^3},
    \end{split}
\end{equation}

\begin{equation}
    \begin{split}
        \det \mM^e_4=&8 \beta r_0\frac{16 \gamma^8 \ln^2(\gamma) - 12 \gamma^8 \ln(\gamma) + 4 \gamma^8 + 12 \gamma^6 \ln(\gamma) - 19 \gamma^6 + 24 \gamma^4  \ln(\gamma) }{\gamma^5}\\
        &8 \beta r_0\frac{+ 27 \gamma^4 - 16 \gamma^2 \ln^2(\gamma) - 24 \gamma^2 \ln(\gamma) - 13 \gamma^2 + 1}{\gamma^5},
    \end{split}
\end{equation}

\begin{equation}
    \begin{split}
       \det \mM^e_5=& 32\beta r_0\frac{-4 \gamma^6 \ln(\gamma) + 3 \gamma^6 - 3  \gamma^4 - 3  \gamma^2 + 4  \ln(\gamma) + 3 }{\gamma^3},
    \end{split}
\end{equation}

\begin{equation}
    \begin{split}
        \det \mM^e_6=&32 \beta r_0\frac{ \gamma^4 \ln(\gamma) -  \gamma^4 + 2 \gamma^2  -  \ln(\gamma) - 1}{\gamma^3}.
    \end{split}
\end{equation}

Asymptotically, we know that
%\begin{equation*}
 %   \begin{split}
      %  \left(\det \mM^e\right)^{-1}=\frac{\gamma^5}{24}\left(1+\frac{32}{3}\gamma^2 (\ln \gamma)^2+\frac{40}{3}\gamma^2 \ln \gamma+\frac{32}{3}\gamma^2+O(\gamma^4 \ln \gamma)\right)
  %    \det \mM^e_1&=-\frac{4\beta r_0}{\gamma}\left(1+O(\gamma^2 \ln \gamma)\right),\\
   %   \det \mM^e_2&=\frac{8\beta r_0}{\gamma^3}\left(4\ln \gamma-1+12\gamma^2+O(\gamma^4 \ln \gamma)\right),\\
    %  \det \mM^e_3&=-\frac{32\beta r_0}{\gamma^3}\left(1-4\gamma^2+O(\gamma^4)\right),\\
     % \det \mM^e_4&=\frac{12\beta r_0}{\gamma^5} \left(1-\frac{32}{3}\gamma^2 (\ln\gamma)^2 +O(\gamma^2 \ln\gamma) \right),\\
     % \det \mM^e_5&=\frac{64\beta r_0}{\gamma^3}\left(2\ln \gamma+1-\frac{3}{4}\gamma^2+O(\gamma^4)\right),\\
      %\det \mM^e_6&=-\frac{24\beta r_0}{\gamma^3}\left(\frac{4}{3}\ln \gamma+1-2\gamma^2+O(\gamma^4\ln \gamma)\right).
    %\end{split}
%\end{equation*}

%Therefore, we know that the solution of (\ref{eq:A}) has the following expression:
\begin{equation}
\begin{split}
A_1&=\frac{\det \mM^e_1}{\det \mM^e}=\beta r_0\left(-\frac{\gamma^4}{2}\right)+O(\beta\gamma^6\ln \gamma),\\
A_2&=\frac{\det \mM^e_2}{\det \mM^e}=2\beta r_0\gamma^2\ln \gamma+O(\beta\gamma^4(\ln \gamma)^3),\\
A_3&=\frac{\det \mM^e_3}{\det \mM^e}=-2\beta r_0\gamma^2\left(1+16\gamma^2(\ln \gamma)^2\right)+O(\beta\gamma^4\ln \gamma),\\
A_4&=\frac{\det \mM^e_4}{\det \mM^e}=\frac{\beta r_0}{2}\left(1+4\gamma^2\right)+O(\beta\gamma^4(\ln \gamma)^4),\\
A_5&=\frac{\det \mM^e_5}{\det \mM^e}=2\beta r_0\gamma^2\left(4\ln \gamma+3\right)+O(\beta\gamma^4(\ln \gamma)^3),\\
A_6&=\frac{\det \mM^e_6}{\det \mM^e}=-2\beta r_0\gamma^2 \left(\ln \gamma+1\right)+O(\beta\gamma^4(\ln \gamma)^3).
\end{split}
\end{equation}
\end{proof}

\begin{lemma}\label{corfficient-f}
    Coefficients $\{B^l_i\}_{l=1}^6$ satisfy
    \begin{equation*}
    \begin{bmatrix}
        B_i^1\\
        B_i^2\\
        B_i^3\\
        B_i^4\\
        B_i^5\\
        B_i^6
    \end{bmatrix}=\alpha^2 q^{(1)}_i\begin{bmatrix}
        O(\gamma^4(\ln \gamma)^{-1})\\
        \frac{1}{6}-\frac{3}{8(\ln \gamma)^2}+O(\gamma^2)\\
        -\frac{1}{2}+O((\ln \gamma)^{-1})\\
        \frac{1}{2\ln \gamma}\left(1-\frac{3}{2\ln \gamma}\right)+O( \gamma^2 (\ln \gamma)^{-1})\\
        \frac{1}{2}\left(1-(\ln \gamma)^{-1}\right)+O((\ln \gamma)^{-2})\\
        -\frac{1}{6}\left(1+\frac{3}{4\ln \gamma}\right)+O(\gamma^2 (\ln \gamma)^{-2})
    \end{bmatrix}
    +\beta^2s^{(1)}_i \begin{bmatrix}
        0\\
        %\frac{\gamma^4}{2}+O(\gamma^4(\ln \gamma)^{-1})\\
       %-\frac{3\gamma^2}{8}\left((\ln \gamma)^{-1}-\frac{3}{2}(\ln \gamma)^{-2}+4\gamma^2\right)+O(\gamma^4(\ln \gamma)^{-1})\\
       0\\
       %-\frac{\gamma^2}{2}\left((\ln \gamma)^{-1}-\frac{3}{2}(\ln \gamma)^{-2}-3\gamma^2(\ln \gamma)^{-1}\right)+O(\gamma^4(\ln \gamma)^{-2})\\
       0\\
       0\\
       %-\frac{1}{12}\left(-2+\frac{9}{2}(\ln \gamma)^{-2}-6\gamma^2(\ln \gamma)^{-1}-27\gamma^2(\ln \gamma)^{-2}+18\gamma^4\right)+O(\gamma^4(\ln \gamma)^{-1})\\
       \frac{1}{6}+O((\ln \gamma)^{-2})\\
      %-\frac{\gamma^2}{8}\left((\ln \gamma)^{-1}-\frac{3}{2}(\ln \gamma)^{-2}+4\gamma^2\right)+O(\gamma^4(\ln \gamma)^{-1})
      0
    \end{bmatrix}.
\end{equation*}
\end{lemma}
\begin{proof}
Comparing each $X_i$, coefficients $\{B_i^l\}_{l=1}^6$ satisfy
    \begin{equation}\label{eq:B}
   \mM^f\cdot \begin{bmatrix}
        B_i^1\\
        B_i^2\\
        B_i^3\\
        B_i^4\\
        B_i^5\\
        B_i^6
    \end{bmatrix}=\begin{bmatrix}
        \frac{\alpha^2}{6\gamma}q^{(1)}_i\\
         -\frac{\alpha^2}{6\gamma^2}q^{(1)}_i\\
         \frac{\alpha^2}{3\gamma^3}q^{(1)}_i\\
        \frac{\beta^2}{6}s^{(1)}_i\\
        \frac{\beta^2}{2} s^{(1)}_i\\
        \beta^2 s^{(1)}_i
    \end{bmatrix}.
\end{equation}

The determinant of $M^f$ and corresponding cofactors are given by
\begin{equation}\label{determinant-of-M-e}
\begin{split}
\det \mM^f=& 128\frac{(-2\gamma^{12} \ln\gamma + 3\gamma^{12} - 12\gamma^{10} + 18\gamma^8 \ln\gamma + 15\gamma^8 )}{\gamma^6}\\
&128\frac{(- 32\gamma^6 \ln\gamma + 18\gamma^4 \ln\gamma - 15\gamma^4 + 12\gamma^2 - 2\ln\gamma - 3)}{\gamma^6},
\end{split}
\end{equation}
\begin{equation}
    \begin{split}
        \det \mM^f_1=&96\alpha^{2}q^{(1)}_i\frac{\gamma^{6} - 4\gamma^{4}\ln(\gamma) - \gamma^{4} + 4\gamma^{2}\ln(\gamma) - \gamma^{2} + 1}{3\gamma^2}\\
       &+ 32\beta^{2}s^{(1)}_i\frac{- 4\gamma^{8}\ln(\gamma) + 3\gamma^{8} - 3\gamma^{6} - 3\gamma^{4} + 4\gamma^{2}\ln(\gamma) + 3\gamma^{2}}{3\gamma^{2}},
    \end{split}
\end{equation}
\begin{equation}
    \begin{split}
       \det \mM^f_2=& 64\alpha^{2}q^{(1)}_i\frac{-6\gamma^{10} + 18\gamma^{8}\ln(\gamma) + 9\gamma^{8} - 16\gamma^{6}\ln(\gamma) - 6\gamma^{4} + 6\gamma^{2} - 2\ln(\gamma) - 3}{3\gamma^6}\\
       &192\beta^{2}s^{(1)}_i\frac{+ 2\gamma^{12}\ln(\gamma) - \gamma^{12} + 2\gamma^{8} - 2\gamma^{4}\ln(\gamma) - \gamma^{4}}{3\gamma^{6}},
    \end{split}
\end{equation}
\begin{equation}
    \begin{split}
  \det \mM^f_3=&-32\alpha^{2}q^{(1)}_i\frac{-3\gamma^{10} + 12\gamma^{8}\ln(\gamma) + 3\gamma^{8} - 8\gamma^{6}\ln(\gamma) + 3\gamma^{2} - 4\ln(\gamma) - 3}{\gamma^6}\\
  &-32\beta^{2}s^{(1)}_i\frac{+ 4\gamma^{12}\ln(\gamma) - 3\gamma^{12} + 3\gamma^{10} + 8\gamma^{6}\ln(\gamma) - 12\gamma^{4}\ln(\gamma) + 3\gamma^{4} - 3\gamma^{2}}{\gamma^{6}},      
    \end{split}
\end{equation}
\begin{equation}
    \begin{split}
        \det \mM^f_4=&128\alpha^{2}q^{(1)}_i\frac{-\gamma^{10} + 3\gamma^{8} - 2\gamma^{6} - 2\gamma^{4} + 3\gamma^{2} - 1}{\gamma^6}\\
        &128\beta^{2}s^{(1)}_i\frac{+\gamma^{12} - 3\gamma^{10} + 2\gamma^{8} + 2\gamma^{6} - 3\gamma^{4} + \gamma^{2}}{\gamma^{6}},
    \end{split}
\end{equation}
\begin{equation}
    \begin{split}
        \det \mM^f_5=&192\alpha^{2}q^{(1)}_i\frac{2\gamma^{8}\ln(\gamma) - \gamma^{8} + 2\gamma^{4} - 2\ln(\gamma) - 1}{3\gamma^{6}}\\
        &64\beta^{2}s^{(1)}_i\frac{- 6\gamma^{10} + 18\gamma^{8}\ln(\gamma) + 9\gamma^{8} - 16\gamma^{6}\ln(\gamma) - 6\gamma^{4} + 6\gamma^{2} - 2\ln(\gamma) - 3}{3\gamma^{6}},
    \end{split}
\end{equation}
\begin{equation}
    \begin{split}
        \det \mM^f_6=&32\alpha^{2}q^{(1)}_i\frac{-4\gamma^{6}\ln(\gamma) + 3\gamma^{6} - 3\gamma^{4} - 3\gamma^{2} + 4\ln(\gamma) + 3}{3\gamma^{6}}\\
        &96\beta^{2}s^{(1)}_i\frac{+ \gamma^{8} - 4\gamma^{6}\ln(\gamma) - \gamma^{6} + 4\gamma^{4}\ln(\gamma) - \gamma^{4} + \gamma^{2}}{3\gamma^{6}}.
    \end{split}
\end{equation}

Asymptotically, we know that
\begin{equation}
\begin{split}
B_i^1&=\frac{\det \mM^f_1}{\det \mM^f}=O(\alpha^2\gamma^4(\ln \gamma)^{-1}),\\
B_i^2&=\frac{\det \mM^f_2}{\det \mM^f}=\alpha^{2}q^{(1)}_i\left(\frac{1}{6}-\frac{3}{8(\ln \gamma)^2}\right)+O(\alpha^2 \gamma^{2}),\\
%\beta^2s^{(1)}_i \frac{\gamma^4}{2}+O(\gamma^4(\ln \gamma)^{-1}),
B_i^3&=\frac{\det \mM^f_3}{\det \mM^f}=-\frac{\alpha^{2}q^{(1)}_i}{2}+O(\alpha^2(\ln \gamma)^{-1}),\\
%\beta^2s^{(1)}_i \left(-\frac{3\gamma^2}{8}\right)\left((\ln \gamma)^{-1}-\frac{3}{2}(\ln \gamma)^{-2}+4\gamma^2\right)+O(\gamma^4(\ln \gamma)^{-1}),
B_i^4&=\frac{\det \mM^f_4}{\det \mM^f}=\frac{\alpha^{2}q^{(1)}_i}{2\ln \gamma}\left(1-\frac{3}{2\ln \gamma}\right)+O(\alpha^2 \gamma^2 (\ln \gamma)^{-1}),\\
%=\beta^2s^{(1)}_i \left(-\frac{\gamma^2}{2}\right)\left((\ln \gamma)^{-1}-\frac{3}{2}(\ln \gamma)^{-2}-3\gamma^2(\ln \gamma)^{-1}\right)+O(\gamma^4(\ln \gamma)^{-2}),
B_i^5&=\frac{\det \mM^f_5}{\det \mM^f}=\frac{\alpha^{2}q^{(1)}_i}{2}\left(1-(\ln \gamma)^{-1}\right)+\frac{\beta^2s^{(1)}_i}{6}+O(\alpha^2(\ln \gamma)^{-2},\beta^2(\ln \gamma)^{-2}),\\
%=\beta^2s^{(1)}_i \left(-\frac{1}{12}\right)\left(-2+\frac{9}{2}(\ln \gamma)^{-2}-6\gamma^2(\ln \gamma)^{-1}-27\gamma^2(\ln \gamma)^{-2}+18\gamma^4\right)+O(\gamma^4(\ln \gamma)^{-1}),
B_i^6&=\frac{\det \mM^f_6}{\det \mM^f}=-\frac{\alpha^{2}q^{(1)}_i}{6}\left(1+\frac{3}{4\ln \gamma}\right)+O(\alpha^2(\ln \gamma)^{-2}).
%=\beta^2s^{(1)}_i \left(-\frac{\gamma^2}{8}\right)\left((\ln \gamma)^{-1}-\frac{3}{2}(\ln \gamma)^{-2}+4\gamma^2\right)+O(\gamma^4(\ln \gamma)^{-1}).
\end{split}
\end{equation}

\end{proof}

\begin{lemma}\label{corfficient-g}
    Coefficients $\mathcal{C}=\{C^l_i\}_{l=1}^6$ satisfy
    \begin{equation*}
    \begin{bmatrix}
        C^1_i\\
        C^2_i\\
        C^3_i\\
        C^4_i\\
        C^5_i\\
        C^6_i
    \end{bmatrix}=\begin{bmatrix}
        O(\alpha\gamma^6,\beta\gamma^6)\\
        \frac{3}{2}\gamma^4\left( -3 \alpha p_i+\beta r_i\right)+O(\alpha\gamma^6)\\
        \frac{1}{2}\left[\alpha p_i(1+9\gamma^2)+\beta r_i(-3\gamma^2)\right]+O(\alpha\gamma^4)\\
        \frac{1}{2}\left[\alpha p_i(-3)+\beta r_i(1+9\gamma^2)\right]+O(\alpha\gamma^4,\beta\gamma^4)\\
        \frac{3}{2}\left[\alpha p_i(1+9\gamma^2)+\beta r_i(-3\gamma^2)\right]+O(\alpha\gamma^4,\beta\gamma^4)\\
       \frac{1}{2}\left[\alpha p_i(-1-9\gamma^2)+\beta r_i(3\gamma^2)\right]+O(\beta\gamma^4)
    \end{bmatrix}.
\end{equation*}
\end{lemma}
\begin{proof}
Comparing each $Y_i$, coefficients $\{C^l_i\}_{l=1}^6$ satisfy the equation
\begin{equation}\label{eq:C}
    \mM^{g}\cdot \begin{bmatrix}
      C^1_i\\
        C^2_i\\
        C^3_i\\
        C^4_i\\
        C^5_i\\
        C^6_i
    \end{bmatrix}=\begin{bmatrix}
        \frac{\alpha}{2}p_i\\
        0\\
        0\\
        \frac{\beta}{2}r_i\\
        \beta r_i\\
        \beta r_i
    \end{bmatrix}.
\end{equation}

The determinant of $M^g$ and corresponding cofactors are given by
\begin{equation}\label{determinant-of-M-e}
\begin{split}
\det \mM^g=&\quad -256\gamma^9 + 2304\gamma^7 - 9216\gamma^5 + 21504\gamma^3 - 32256\gamma + \frac{32256}{\gamma} \\
&\quad - \frac{21504}{\gamma^3} + \frac{9216}{\gamma^5} - \frac{2304}{\gamma^7} + \frac{256}{\gamma^9},
\end{split}
\end{equation}
\begin{equation}
\begin{split}
\det \mM^g_1=& 384\alpha p_i \frac{ \gamma^{10} - 3\gamma^8 + 2\gamma^6 + 2\gamma^4 - 3\gamma^2 + 1}{\gamma^3} \\
      &128\beta r_i \frac{- \gamma^{12}  + 9\gamma^8- 16\gamma^6  + 9\gamma^4 - 1
  }{\gamma^3} 
\end{split}
\end{equation}
\begin{equation}
\begin{split}
\det \mM^g_2=& 384\alpha p_i \frac{
      -3\gamma^{12} + 8\gamma^{10} - 5\gamma^8 - 5\gamma^4 + 8\gamma^2 - 3}{\gamma^5} \\
      &384 \beta r_i\frac{ + \gamma^{14} - 6\gamma^{10}+ 5\gamma^8 + 5\gamma^6 - 6\gamma^4  +1  }{\gamma^5} 
\end{split}
\end{equation}
\begin{equation}
\begin{split}
\det \mM^g_3=& 128\alpha p_i \frac{
    9\gamma^{16} - 18\gamma^{14} + 10\gamma^{12} - 36\gamma^{10} + 90\gamma^8  - 74\gamma^6 + 18\gamma^4 + 1}{\gamma^9} \\
      &384 \beta r_i\frac{- \gamma^{18}  + 16\gamma^{12}  - 30\gamma^{10} + 16\gamma^8 - \gamma^2   }{\gamma^9} 
\end{split}
\end{equation}
\begin{equation}
\begin{split}
\det \mM^g_4=& 384\alpha p_i\frac{
      -\gamma^{16} + 16\gamma^{10} - 30\gamma^8 + 16\gamma^6 - 1}{\gamma^9} \\
      &128\beta r_i \frac{+ 9\gamma^{16}  - 18\gamma^{14} + 10\gamma^{12} - 36\gamma^{10} + 90\gamma^8  - 74\gamma^6  + 18\gamma^4  + 1  }{\gamma^9} 
\end{split}
\end{equation}

\begin{equation}
\begin{split}
\det \mM^g_5=&384\alpha p_i \frac{
      \gamma^{14} - 6\gamma^{10} + 5\gamma^8 + 5\gamma^6 - 6\gamma^4  + 1}{\gamma^9} \\
      &384 \beta r_i \frac{- 3\gamma^{14} + 8\gamma^{12} - 5\gamma^{10} - 5\gamma^6 + 8\gamma^4 - 3\gamma^2 }{\gamma^9} 
\end{split}
\end{equation}
\begin{equation}
    \begin{split}
        \det \mM^g_6=&128\alpha p_i \frac{
      -\gamma^{12}  + 9\gamma^8  - 16\gamma^6 + 9\gamma^4 - 1}{\gamma^9} \\
      &384\beta r_i\frac{+ \gamma^{12} - 3\gamma^{10} + 2\gamma^8 + 2\gamma^6 - 3\gamma^4 + \gamma^2  }{\gamma^9}
    \end{split}
\end{equation}
Asymptotically, we know that
\begin{equation}
\begin{split}
C_i^1&=\frac{\det \mM^g_1}{\det \mM^g}=O(\alpha\gamma^6, \beta\gamma^6),\\
C_i^2&=\frac{\det \mM^g_2}{\det \mM^g}=\frac{3}{2}\gamma^4\left( -3 \alpha p_i+\beta r_i\right)+O(\alpha\gamma^6),\\
C_i^3&=\frac{\det \mM^g_3}{\det \mM^g}=\frac{1}{2}\left[\alpha p_i(1+9\gamma^2)+\beta r_i(-3\gamma^2)\right]+O(\alpha\gamma^4),\\
C_i^4&=\frac{\det \mM^g_4}{\det \mM^g}=\frac{1}{2}\left[\alpha p_i(-3)+\beta r_i(1+9\gamma^2)\right]+O(\alpha\gamma^4,\beta\gamma^4),\\
C_i^5&=\frac{\det \mM^g_5}{\det \mM^g}=\frac{3}{2}\left[\alpha p_i(1+9\gamma^2)+\beta r_i(-3\gamma^2)\right]+O(\alpha\gamma^4, \beta\gamma^4),\\
C_i^6&=\frac{\det \mM^g_6}{\det \mM^g}=\frac{1}{2}\left[\alpha p_i(-1-9\gamma^2)+\beta r_i(3\gamma^2)\right]+O(\alpha\gamma^4,\beta\gamma^4).
\end{split}
\end{equation}

\end{proof}

\begin{lemma}\label{corfficient-h}
    Coefficients $\{D^l_i\}_{l=1}^6$ satisfy
    \begin{equation*}
    \begin{bmatrix}
        D^1_i\\
        D^2_i\\
        D^3_i\\
        D^4_i\\
        D^5_i\\
        D^6_i
    \end{bmatrix}=\begin{bmatrix}
        \left(3\alpha^2q^{(3)}_i-\frac{1}{2}\beta^2 s^{(3)}_i\right)\gamma^8+O(\alpha^2\gamma^{10})\\
       \left(-8\alpha^2q^{(3)}_i+\frac{4}{3}\beta^2 s^{(3)}_i\right)\gamma^6+O(\alpha^2\gamma^{8})\\
        \frac{1}{6}\alpha^2 q^{(3)}_i\left(1+36\gamma^4\right)-\beta^2 s^{(3)}_i\gamma^4+O(\alpha^2\gamma^6 )\\
        \left(-\alpha^2 q^{(3)}_i+\frac{1}{6}\beta^2 s^{(3)}_i\right)\left(1+36\gamma^4\right)+O(\alpha^2\gamma^6 )\\
        \frac{4}{3}\alpha^2 q^{(3)}_i\left(1+36\gamma^4\right)-8\beta^2 s^{(3)}_i\gamma^4+O(\alpha^2\gamma^6,\beta^2\gamma^6 )\\
       -\frac{1}{2}\alpha^2 q^{(3)}_i\left(1+36\gamma^4\right)+3\beta^2 s^{(3)}_i\gamma^4+O(\alpha^2\gamma^6,\beta^2\gamma^6 )
    \end{bmatrix}.
\end{equation*}
\end{lemma}
\begin{proof}
Comparing each $W_i$, coefficients $\{D^l_i\}_{l=1}^6$ satisfy the equation
\begin{equation}\label{eq:D}
    \mM^{h}\cdot \begin{bmatrix}
      D^1_i\\
        D^2_i\\
        D^3_i\\
        D^4_i\\
        D^5_i\\
        D^6_i
    \end{bmatrix}=\begin{bmatrix}
       \frac{\alpha^2}{6\gamma}q^{(3)}_i\\
        -\frac{\alpha^2}{6\gamma^2}q^{(3)}_i\\
        \frac{\alpha^2}{3\gamma^3}q^{(3)}_i\\
        \frac{\beta^2}{6}s^{(3)}_i\\
        \frac{\beta^2}{2}s^{(3)}_i\\
       \beta^2s^{(3)}_i
    \end{bmatrix}.
\end{equation}

The determinant of $M^h$ and corresponding cofactors are given by
\begin{equation}\label{determinant-of-M-e}
\begin{split}
\det \mM^h=& \frac{256(-\gamma^{24} + 36\gamma^{20} - 160\gamma^{18} + 315\gamma^{16} - 288\gamma^{14} + 288\gamma^{10} - 315\gamma^8 + 160\gamma^6 - 36\gamma^4 + 1)}{\gamma^{12}},
\end{split}
\end{equation}
\begin{equation}
\begin{split}
\det \mM^h_1=&-256\alpha^2 q^{(3)}_i \frac{ -3\gamma^{12} + 8\gamma^{10}  - 5\gamma^8  - 5\gamma^4  + 8\gamma^2 - 3}{\gamma^4}\\
&-128\beta^2 s^{(3)}_i\frac{+ \gamma^{16} - 16\gamma^{10}+ 30\gamma^8 - 16\gamma^6 + 1 }{\gamma^4},
\end{split}
\end{equation}
\begin{equation}
\begin{split}
\det \mM^h_2=&-3072\alpha^2 q^{(3)}_i \frac{
      2\gamma^{14} - 5\gamma^{12} + 3\gamma^{10} + 3\gamma^4 - 5\gamma^2  + 2}{3\gamma^6}\\
      &-1024\beta^2 s^{(3)}_i\frac{- \gamma^{18} + 10\gamma^{12}  - 9\gamma^{10} - 9\gamma^8  + 10\gamma^6  - 1}{3\gamma^6},
\end{split}
\end{equation}
\begin{equation}
\begin{split}
\det \mM^h_3=&128\alpha^2 q^{(3)}_i \frac{
    36\gamma^{20} - 80\gamma^{18} + 45\gamma^{16}  - 100\gamma^{12} + 288\gamma^{10}  - 270\gamma^8 + 80\gamma^6  + 1}{3\gamma^{12}}\\
    & 768\beta^2 s^{(3)}_i\frac{- \gamma^{24}  + 25\gamma^{16} - 48\gamma^{14} + 25\gamma^{12} - \gamma^4
  }{3\gamma^{12}},
\end{split}
\end{equation}
\begin{equation}
\begin{split}
\det \mM^h_4=& 768\alpha^2 q^{(3)}_i\frac{
      -\gamma^{20}  + 25\gamma^{12} - 48\gamma^{10} +25\gamma^8  - 1 }{3\gamma^{12}}  \\
      &128\beta^2 s^{(3)}_i\frac{+ 36\gamma^{20}  - 80\gamma^{18}  + 45\gamma^{16}  - 100\gamma^{12}  + 288\gamma^{10}  - 270\gamma^8  + 80\gamma^6  + 1
  }{3\gamma^{12}},  
\end{split}
\end{equation}

\begin{equation}
\begin{split}
\det \mM^h_5=&1024\alpha^2 q^{(3)}_i\frac{
      \gamma^{18} - 10\gamma^{12}  + 9\gamma^{10} + 9\gamma^8  - 10\gamma^6  + 1 }{3\gamma^{12}} \\
      &3072\beta^2 s^{(3)}_i \frac{- 2\gamma^{18}  + 5\gamma^{16} - 3\gamma^{14} - 3\gamma^8  + 5\gamma^6  - 2\gamma^4  }{3\gamma^{12}},
\end{split}
\end{equation}
\begin{equation}
    \begin{split}
        \det \mM^h_6=&128\alpha^2 q^{(3)}_i \frac{
      -\gamma^{16}  + 16\gamma^{10} - 30\gamma^8  + 16\gamma^6  - 1}{\gamma^{12}}\\
      &256\beta^2 s^{(3)}_i \frac{+ 3\gamma^{16}  - 8\gamma^{14}  + 5\gamma^{12}   + 5\gamma^8  - 8\gamma^6  + 3\gamma^4  }{\gamma^{12}},
    \end{split}
\end{equation}
Asymptotically, we know that
\begin{equation}
\begin{split}
D_i^1&=\frac{\det \mM^h_1}{\det \mM^h}=\left(3\alpha^2q^{(3)}_i-\frac{1}{2}\beta^2 s^{(3)}_i\right)\gamma^8+O(\alpha^2\gamma^{10}),\\
D_i^2&=\frac{\det \mM^h_2}{\det \mM^h}=\left(-8\alpha^2q^{(3)}_i+\frac{4}{3}\beta^2 s^{(3)}_i\right)\gamma^6+O(\alpha^2\gamma^{8}),\\
D_i^3&=\frac{\det \mM^h_3}{\det \mM^h}=\frac{1}{6}\alpha^2 q^{(3)}_i\left(1+36\gamma^4\right)-\beta^2 s^{(3)}_i\gamma^4+O(\alpha^2\gamma^6),\\
D_i^4&=\frac{\det \mM^h_4}{\det \mM^h}=\left(-\alpha^2 q^{(3)}_i+\frac{1}{6}\beta^2 s^{(3)}_i\right)\left(1+36\gamma^4\right)+O(\beta^2\gamma^6),\\
D_i^5&=\frac{\det \mM^h_5}{\det \mM^h}=\frac{4}{3}\alpha^2 q^{(3)}_i\left(1+36\gamma^4\right)-8\beta^2 s^{(3)}_i\gamma^4+O(\alpha^2\gamma^6,\beta^2\gamma^6 ),\\
D_i^6&=\frac{\det \mM^h_6}{\det \mM^h}=-\frac{1}{2}\alpha^2 q^{(3)}_i\left(1+36\gamma^4\right)+3\beta^2 s^{(3)}_i\gamma^4+O(\alpha^2\gamma^6,\beta^2\gamma^6 ).
\end{split}
\end{equation}

\end{proof}

\section{Bilinear forms}\label{bilinear}
For functions $u,v:\bB_{\sigma}(0)\backslash \bB_{\tau}(0)\to \bR^{n-4}$, we consider the bilinear form
\begin{equation*}
    \mB_{\sigma, \tau}(u,v):=\int_{\bB_{\sigma}(0)\backslash \bB_{\tau}(0)} \langle \nabla \Delta_0 u, \nabla \Delta_0 v\rangle \dvol.
\end{equation*}
We start by computing $\mB_{\sigma, \tau}(\cdot,\cdot)$ with respect to the bases in Lemma \ref{e}, \ref{f}, \ref{g} and \ref{h}.
\begin{lemma}\label{energy-e}
Let $E(r)$ be the vector defined in Lemma \ref{e}. Then
    \begin{equation}
\begin{split}
    &\int_{\bB_{\sigma}(0)\backslash \bB_{\tau}(0)}\langle \nabla \Delta_0 (E(r)), \nabla \Delta_0(E(r))\rangle \dvol\\
=&16\Vol(\bS^3)\begin{bmatrix}
0 & 0 & 0 & 0 & 0& 0\\
0 & 0 & 0 & 0 & 0& 0\\
0 & 0 & -\frac{r^{-2}}{2} & 0 & -2\ln r& -6r^2\\
0 & 0 & 0 & 0 & 0& 0\\
0 & 0 & -2\ln r & 0 & 2r^2 & 6r^4\\
0 & 0 & -6r^{2} & 0 & 6r^4 & 24r^6
\end{bmatrix}\Bigg|^{r=\sigma}_{r=\tau}.
    \end{split}
\end{equation}
\end{lemma}
\begin{proof}
    From Lemma \ref{e}, we already know that the Laplace operator $\Delta_0$ acts on the six-dimensional space $E(r)$ as the matrix:
\begin{equation}
    \Delta_0(E(r))=E(r)\cdot 
    \begin{bmatrix}
0 & 0 & 2 & 0 & 0& 0\\
0 & 0 & 0 & 8 & 6& 0\\
0 & 0 & 0 & 0 & 8& 0\\
0 & 0 & 0 & 0 & 0& 24\\
0 & 0 & 0 & 0 & 0& 0\\
0 & 0 & 0 & 0 & 0& 0
\end{bmatrix}.
\end{equation}
Note that acting on radial functions, we have $\nabla=\frac{x}{r}\frac{\partial }{\partial r}$, and
\begin{equation*}
        \partial_r E(r)= r^{-1}E(r)\cdot\begin{bmatrix}
-2 & 0 & 0 & 0 & 0& 0\\
0 & 0 & 1 & 0 & 0& 0\\
0 & 0 & 0 & 0 & 0& 0\\
0 & 0 & 0 & 2 & 1& 0\\
0 & 0 & 0 & 0 & 2& 0\\
0 & 0 & 0 & 0 & 0& 4
        \end{bmatrix}.
    \end{equation*}

Consequently, we have
\begin{equation}
    \partial_r \circ \Delta_0(E(r))=r^{-1}\cdot E(r)\cdot 
    \begin{bmatrix}
0 & 0 & -4 & 0 & 0& 0\\
0 & 0 & 0 & 0 & 8& 0\\
0 & 0 & 0 & 0 & 0& 0\\
0 & 0 & 0 & 0 & 0& 48\\
0 & 0 & 0 & 0 & 0& 0\\
0 & 0 & 0 & 0 & 0& 0
\end{bmatrix}.
\end{equation}

Therefore,
\begin{equation}
\begin{split}
    &\int_{\bB_{\sigma}(0)\backslash \bB_{\tau}(0)}\langle \nabla \Delta_0 (E(r)), \nabla \Delta_0(E(r))\rangle \dvol\\
    =&16\Vol(\bS^3)\int^{\sigma}_{\tau}  \begin{bmatrix}
0 & 0 & 0 & 0 & 0& 0\\
0 & 0 & 0 & 0 & 0& 0\\
0 & 0 & r^{-4} & 0 & -2r^{-2}& -12\\
0 & 0 & 0 & 0 & 0& 0\\
0 & 0 & -2r^{-2} & 0 & 4& 24r^2\\
0 & 0 & -12 & 0 & 24r^2& 144r^4
\end{bmatrix}rdr\\
=&16\Vol(\bS^3)\begin{bmatrix}
0 & 0 & 0 & 0 & 0& 0\\
0 & 0 & 0 & 0 & 0& 0\\
0 & 0 & -\frac{r^{-2}}{2} & 0 & -2\ln r& -6r^2\\
0 & 0 & 0 & 0 & 0& 0\\
0 & 0 & -2\ln r & 0 & 2r^2 & 6r^4\\
0 & 0 & -6r^{2} & 0 & 6r^4 & 24r^6
\end{bmatrix}\Bigg|^{r=\sigma}_{r=\tau}.
    \end{split}
\end{equation}
\end{proof}

To compute $\mB_{\sigma, \tau}(\cdot,\cdot)$ with respect to $F(r)\cdot X_i$, $G(r)\cdot Y_i$ and $H(r)\cdot W_i$, we apply the integration by parts to obtain
\begin{equation}\label{integration-by-parts}
\begin{split}
    \mB_{\sigma, \tau}(u,v)=&-\int_{B_{\sigma}(0)\backslash B_{\tau}(0)} \langle  \Delta_0 u, \Delta^2_0 v\rangle \dvol\\
    &+\int_{\partial B_{\sigma}}\langle \Delta_0 u, \partial_r \Delta_0 v\rangle dA-\int_{\partial B_{\tau}}\langle \Delta_0 u, \partial_r \Delta_0 v\rangle dA.
    \end{split}
\end{equation}
%Note that the identity (\ref{integration-by-parts}) does't not hold for $e_3(r)=\ln r$ due to the singularity of $r^{-2}$ at the origin. Hence, we compute $\mB_{\sigma, \tau}(E(r),E(r))$ directly.

\begin{lemma}\label{energy-f}
    Let $F(r)$ be the vector defined in Lemma \ref{f}. Then for each $X_i\in \mathcal{H}^{\bS}_{1}$, we have
\begin{equation}
\begin{split}
    &\int_{\bB_{\sigma}(0)\backslash \bB_{\tau}(0)}\langle \nabla \Delta_0 (F(r)\cdot X_i), \nabla \Delta_0(F(r)\cdot X_i)\rangle \dvol\\
    =&16 \Vol(\bS^3) \begin{bmatrix}
0 & 0 & 0 & 0 & 0& 0\\
0 & -3r^{-4} & 0 & 3r^{-2} & 0& 24r^2\\
0 & 0 & 0& 0& 0 & 0\\
0 & 3r^{-2} & 0 & {4\ln r} & 3r^2& 0\\
0 & 0 & 0 & 3r^2 & 9r^4& 24r^6\\
0 & 24r^2 & 0 & 0 & 24r^6& 96r^8
\end{bmatrix}\Bigg|_{r=\tau}^{r=\sigma}.
\end{split}
\end{equation}
    
\end{lemma}
\begin{proof}
    From Lemma \ref{f}, we already know that 
    the Laplace operator $\Delta_0$ acts on the six-dimensional space $F(r)\cdot X_i$ as the matrix:
\begin{equation}
    \Delta_0(F(r)\cdot X_i)=F(r)\cdot X_i\cdot 
    \begin{bmatrix}
0 & -4 & 0 & 0 & 0& 0\\
0 & 0 & 0 & 4 & 0& 0\\
0 & 0 & 0 & 0 & 12& 0\\
0 & 0 & 0 & 0 & 0& 0\\
0 & 0 & 0 & 0 & 0& 32\\
0 & 0 & 0 & 0 & 0& 0
\end{bmatrix},
\end{equation}
and
\begin{equation*}
        \partial_r \left(F(r)\cdot X_i\right)=r^{-1}F(r)\cdot X_i\cdot\begin{bmatrix}
-3 & 0 & 0 & 0 & 0& 0\\
0 & -1 & 0 & 0 & 0& 0\\
0 & 0 & 1 & 1 & 0& 0\\
0 & 0 & 0 & 1 & 0& 0\\
0 & 0 & 0 & 0 & 3& 0\\
0 & 0 & 0 & 0 & 0& 5
        \end{bmatrix}.
    \end{equation*}

The bi-Laplace operator $\Delta^2_0$ acts on the six-dimensional space $F(r)$ as the matrix:
\begin{equation}
    \Delta^2_0(F(r)\cdot X_i)=F(r)\cdot X_i\cdot 
    \begin{bmatrix}
0 & 0 & 0 & -16 & 0& 0\\
0 & 0 & 0 & 0 & 0& 0\\
0 & 0 & 0 & 0 & 0& 384\\
0 & 0 & 0 & 0 & 0& 0\\
0 & 0 & 0 & 0 & 0& 0\\
0 & 0 & 0 & 0 & 0& 0
\end{bmatrix}.
\end{equation}

Therefore, we have
\begin{equation}
\begin{split}
    &-\int_{\bB_{\sigma}(0)\backslash \bB_{\tau}(0)}\langle \Delta_0 (F(r)\cdot X_i), \Delta^2_0(F(r)\cdot X_i)\rangle \dvol\\
=&-64\int_{\bS^3} X^2_i\dvol_{g_c}\int_{\tau}^{\sigma}\begin{bmatrix}
0 & 0 & 0 & 0 & 0& 0\\
0 & 0 & 0 & r^{-6} & 0& -24r^{-2}\\
0 & 0 & 0 & 0 & 0& 0\\
0 & 0 & 0 & -r^{-4} & 0& 24\\
0 & 0 & 0 & -3r^{-2} & 0& 72r^2\\
0 & 0 & 0 & -8 & 0 & 192r^4
\end{bmatrix}r^3dr\\
=&-64\Vol(\bS^3) \begin{bmatrix}
0 & 0 & 0 & 0 & 0& 0\\
0 & 0 & 0 & -\frac{r^{-2}}{2} & 0& -12r^2\\
0 & 0 & 0 & 0 & 0& 0\\
0 & 0 & 0 & -\ln r & 0& 6r^4\\
0 & 0 & 0 & -\frac{3}{2}r^2 & 0& 12r^6\\
0 & 0 & 0 & -2r^4 & 0 & 24r^8
\end{bmatrix}\Bigg|_{r=\tau}^{r=\sigma}.
    \end{split}
\end{equation}
Besides,
\begin{equation}
    \partial_r \circ \Delta_0(F(r)\cdot X_i)=r^{-1} F(r)\cdot X_i\cdot 
    \begin{bmatrix}
0 & 12 & 0 & 0 & 0& 0\\
0 & 0 & 0 & -4 & 0& 0\\
0 & 0 & 0 & 0 & 12& 0\\
0 & 0 & 0 & 0 & 0& 0\\
0 & 0 & 0 & 0 & 0& 96\\
0 & 0 & 0 & 0 & 0& 0
\end{bmatrix}.
\end{equation}
We have
\begin{equation}
\begin{split}
    &\int_{\partial \bB_{\sigma}(0)\backslash \partial \bB_{\tau}(0)} \langle \Delta_0(F(r)\cdot X_i), \partial_r \circ \Delta_0(F(r)\cdot X_i) \rangle dA\\
    =&16 \Vol(\bS^3) \begin{bmatrix}
0 & 0 & 0 & 0 & 0& 0\\
0 & -3r^{-4} & 0 & r^{-2} & -3& -24r^2\\
0 & 0 & 0& 0& 0 & 0\\
0 & 3r^{-2} & 0 & -1 & 3r^2& 24r^4\\
0 & 9 & 0 & -3r^2 & 9r^4& 72r^6\\
0 & 24r^2 & 0 & -8r^4 & 24r^6& 192r^8
\end{bmatrix}\Big|_{r=\tau}^{r=\sigma}.
    \end{split}
\end{equation}
Combining them together yields
\begin{equation}
\begin{split}
    &\int_{\bB_{\sigma}(0)\backslash \bB_{\tau}(0)}\langle \nabla \Delta_0 (F(r)\cdot X_i), \nabla \Delta_0(F(r)\cdot X_i)\rangle \dvol\\
    =&16 \Vol(\bS^3) \begin{bmatrix}
0 & 0 & 0 & 0 & 0& 0\\
0 & -3r^{-4} & 0 & 3r^{-2} & 0& 24r^2\\
0 & 0 & 0& 0& 0 & 0\\
0 & 3r^{-2} & 0 & {4\ln r} & 3r^2& 0\\
0 & 0 & 0 & 3r^2 & 9r^4& 24r^6\\
0 & 24r^2 & 0 & 0 & 24r^6& 96r^8
\end{bmatrix}\Big|_{r=\tau}^{r=\sigma},
\end{split}
\end{equation}
Here we have replaced constant functions $-3$ and $9$ by zeros which will not affect our results.
\end{proof}

\begin{lemma}\label{energy-g}
    Let $G(r)$ be the vector defined in Lemma \ref{g}. Then for each $Y_i\in \mathcal{H}^{\bS}_{2}$, we have
\begin{equation}
\begin{split}
    &\int_{\bB_{\sigma}(0)\backslash \bB_{\tau}(0)}\langle \nabla \Delta_0 (G(r)\cdot Y_i), \nabla \Delta_0(G(r)\cdot Y_i)\rangle \dvol\\
    =&128 \Vol(\bS^3) \begin{bmatrix}
0 & 0 & 0 & 0 & 0& 0\\
0 & -2r^{-6} & -2r^{-4} & 0 & 0& 10r^2\\
0 & -2r^{-4} & {-3r^{-2}}& 0& -2r^2 & 0\\
0 & 0 & 0 & 0 & 0& 0\\
0 & 0 & -2r^2 & 0 & 4r^6& 10r^8\\
0 & 10r^2 & 0 & 0 & 10r^8& {30r^{10}}
\end{bmatrix}\Bigg|_{r=\tau}^{r=\sigma}.
\end{split}
\end{equation}
    
\end{lemma}
\begin{proof}
    From Lemma \ref{g}, we already know that 
    the Laplace operator $\Delta_0$ acts on the six-dimensional space $G(r)\cdot Y_i$ as the matrix:
\begin{equation}
    \Delta_0(G(r)\cdot Y_i)=G(r)\cdot Y_i\cdot 
    \begin{bmatrix}
0 & -8 & 0 & 0 & 0& 0\\
0 & 0 & -8 & 0 & 0& 0\\
0 & 0 & 0 & 0 & 0& 0\\
0 & 0 & 0 & 0 & 16& 0\\
0 & 0 & 0 & 0 & 0& 40\\
0 & 0 & 0 & 0 & 0& 0
\end{bmatrix},
\end{equation}
and
\begin{equation*}
        \partial_r \left(G(r)\cdot Y_i\right)=r^{-1}G(r)\cdot Y_i\cdot\begin{bmatrix}
-4 & 0 & 0 & 0 & 0& 0\\
0 & -2 & 0 & 0 & 0& 0\\
0 & 0 & 0 & 0 & 0& 0\\
0 & 0 & 0 & 2 & 0& 0\\
0 & 0 & 0 & 0 & 4& 0\\
0 & 0 & 0 & 0 & 0& 6
        \end{bmatrix}.
    \end{equation*}

The bi-Laplace operator $\Delta^2_0$ acts on the six-dimensional space $G(r)$ as the matrix:
\begin{equation}
    \Delta^2_0(G(r)\cdot Y_i)=G(r)\cdot Y_i\cdot 
    64 \begin{bmatrix}
0 & 0 & 1 & 0 & 0& 0\\
0 & 0 & 0 & 0 & 0& 0\\
0 & 0 & 0 & 0 & 0& 0\\
0 & 0 & 0 & 0 & 0& 10\\
0 & 0 & 0 & 0 & 0& 0\\
0 & 0 & 0 & 0 & 0& 0
\end{bmatrix}.
\end{equation}

Therefore, we have
\begin{equation}
\begin{split}
    &-\int_{\bB_{\sigma}(0)\backslash \bB_{\tau}(0)}\langle \Delta_0 (G(r)\cdot Y_i), \Delta^2_0(G(r)\cdot Y_i)\rangle \dvol\\
=&512\int_{\bS^3} Y^2_i\dvol_{g_c}\int_{\tau}^{\sigma}\begin{bmatrix}
0 & 0 & 0 & 0 & 0& 0\\
0 & 0 & r^{-8} & 0 & 0& 10r^{-2}\\
0 & 0 & r^{-6} & 0 & 0& 10\\
0 & 0 & 0 & 0 & 0& 0\\
0 & 0 & -2r^{-2} & 0 & 0& -20r^4\\
0 & 0 & -5 & 0 & 0 & -50r^6
\end{bmatrix}r^3dr\\
=&512\Vol(\bS^3) \begin{bmatrix}
0 & 0 & 0 & 0 & 0& 0\\
0 & 0 & -\frac{r^{-4}}{4} & 0 & 0& 5r^{2}\\
0 & 0 & -\frac{r^{-2}}{2} & 0 & 0& \frac{5}{2}r^4\\
0 & 0 & 0 & 0 & 0& 0\\
0 & 0 & -r^2 & 0 & 0& -\frac{5}{2}r^8\\
0 & 0 & -\frac{5}{4}r^4 & 0 & 0 & -5r^{10}
\end{bmatrix}\Bigg|_{r=\tau}^{r=\sigma}.
    \end{split}
\end{equation}
Besides,
\begin{equation}
    \partial_r \circ \Delta_0(G(r)\cdot Y_i)=16 r^{-1} G(r)\cdot Y_i\cdot 
    \begin{bmatrix}
0 & 2 & 0 & 0 & 0& 0\\
0 & 0 & 1 & 0 & 0& 0\\
0 & 0 & 0 & 0 & 0& 0\\
0 & 0 & 0 & 0 & 2& 0\\
0 & 0 & 0 & 0 & 0& 10\\
0 & 0 & 0 & 0 & 0& 0
\end{bmatrix}.
\end{equation}
We have
\begin{equation}
\begin{split}
    &\int_{\partial \bB_{\sigma}(0)\backslash \partial \bB_{\tau}(0)} \langle \Delta_0(G(r)\cdot Y_i), \partial_r \circ \Delta_0(G(r)\cdot Y_i) \rangle dA\\
    =&-128 \Vol(\bS^3) \begin{bmatrix}
0 & 0 & 0 & 0 & 0& 0\\
0 & 2r^{-6} & r^{-4} & 0 & 2& 10r^2\\
0 & 2r^{-4} & r^{-2}& 0& 2r^2 & 10r^4\\
0 & 0 & 0 & 0 & 0& 0\\
0 & -4 & -2r^2 & 0 & -4r^6& -20r^8\\
0 & -10r^2 & -5r^4 & 0 & -10r^8& -50r^{10}
\end{bmatrix}\Big|_{r=\tau}^{r=\sigma}.
    \end{split}
\end{equation}
Combining them together yields
\begin{equation}
\begin{split}
    &\int_{\bB_{\sigma}(0)\backslash \bB_{\tau}(0)}\langle \nabla \Delta_0 (G(r)\cdot Y_i), \nabla \Delta_0(G(r)\cdot Y_i)\rangle \dvol\\
    =&128 \Vol(\bS^3) \begin{bmatrix}
0 & 0 & 0 & 0 & 0& 0\\
0 & -2r^{-6} & -2r^{-4} & 0 & 0& 10r^2\\
0 & -2r^{-4} & {-3r^{-2}}& 0& -2r^2 & 0\\
0 & 0 & 0 & 0 & 0& 0\\
0 & 0 & -2r^2 & 0 & 4r^6& 10r^8\\
0 & 10r^2 & 0 & 0 & 10r^8& {30r^{10}}
\end{bmatrix}\Bigg|_{r=\tau}^{r=\sigma}.
\end{split}
\end{equation}
\end{proof}

\begin{lemma}\label{energy-h}
    Let $H(r)$ be the vector defined in Lemma \ref{h}. Then for each $W_i\in \mathcal{H}^{\bS}_{3}$, we have
\begin{equation}
\begin{split}
    &\int_{\bB_{\sigma}(0)\backslash \bB_{\tau}(0)}\langle \nabla \Delta_0 (H(r)\cdot W_i), \nabla \Delta_0(H(r)\cdot W_i)\rangle \dvol\\
    =&48 \Vol(\bS^3) \begin{bmatrix}
0 & 0 & 0 & 0 & 0& 0\\
0 & -15r^{-8} & -20r^{-6} & 0 & 0& 60r^2\\
0 & -20r^{-6} & {-32r^{-4}}& 0& -20r^2 & 0\\
0 & 0 & 0 & 0 & 0& 0\\
0 & 0 & -20r^2 & 0 & 25r^8& 60r^{10}\\
0 & 60r^2 & 0 & 0 & 60r^{10}& 240r^{12}
\end{bmatrix}\Big|_{r=\tau}^{r=\sigma}.
\end{split}
\end{equation}
\end{lemma}
\begin{proof}
    From Lemma \ref{h}, we already know that 
    the Laplace operator $\Delta_0$ acts on the six-dimensional space $H(r)\cdot W_i$ as the matrix:
\begin{equation}
    \Delta_0(H(r)\cdot W_i)=4H(r)\cdot W_i\cdot 
    \begin{bmatrix}
0 & -3 & 0 & 0 & 0& 0\\
0 & 0 & -4 & 0 & 0& 0\\
0 & 0 & 0 & 0 & 0& 0\\
0 & 0 & 0 & 0 & 5& 0\\
0 & 0 & 0 & 0 & 0& 12\\
0 & 0 & 0 & 0 & 0& 0
\end{bmatrix},
\end{equation}
and
\begin{equation*}
        \partial_r \left(H(r)\cdot W_i\right)=r^{-1}H(r)\cdot W_i\cdot\begin{bmatrix}
-5 & 0 & 0 & 0 & 0& 0\\
0 & -3 & 0 & 0 & 0& 0\\
0 & 0 & -1 & 0 & 0& 0\\
0 & 0 & 0 & 3 & 0& 0\\
0 & 0 & 0 & 0 & 5& 0\\
0 & 0 & 0 & 0 & 0& 7
        \end{bmatrix}.
    \end{equation*}

The bi-Laplace operator $\Delta^2_0$ acts on the six-dimensional space $H(r)$ as the matrix:
\begin{equation}
    \Delta^2_0(H(r)\cdot W_i)=192H(r)\cdot W_i\cdot 
     \begin{bmatrix}
0 & 0 & 1 & 0 & 0& 0\\
0 & 0 & 0 & 0 & 0& 0\\
0 & 0 & 0 & 0 & 0& 0\\
0 & 0 & 0 & 0 & 0& 5\\
0 & 0 & 0 & 0 & 0& 0\\
0 & 0 & 0 & 0 & 0& 0
\end{bmatrix}.
\end{equation}

Therefore, we have
\begin{equation}
\begin{split}
    &-\int_{\bB_{\sigma}(0)\backslash \bB_{\tau}(0)}\langle \Delta_0 (H(r)\cdot W_i), \Delta^2_0(H(r)\cdot W_i)\rangle \dvol\\
=&-768\int_{\bS^3} W^2_i\dvol_{g_c}\int_{\tau}^{\sigma}\begin{bmatrix}
0 & 0 & 0 & 0 & 0& 0\\
0 & 0 & -3r^{-10} & 0 & 0& -15r^{-2}\\
0 & 0 & -4r^{-8} & 0 & 0& -20\\
0 & 0 & 0 & 0 & 0& 0\\
0 & 0 & 5r^{-2} & 0 & 0& 25r^6\\
0 & 0 & 12 & 0 & 0 & 60r^8
\end{bmatrix}r^3dr\\
=&-768\Vol(\bS^3) \begin{bmatrix}
0 & 0 & 0 & 0 & 0& 0\\
0 & 0 & \frac{r^{-6}}{2} & 0 & 0& -\frac{15}{2}r^{2}\\
0 & 0 & r^{-4} & 0 & 0& -5r^4\\
0 & 0 & 0 & 0 & 0& 0\\
0 & 0 & \frac{5}{2}r^2 & 0 & 0& \frac{5}{2}r^{10}\\
0 & 0 & 3r^4 & 0 & 0 & 5r^{12}
\end{bmatrix}\Bigg|_{r=\tau}^{r=\sigma}.
    \end{split}
\end{equation}
Besides,
\begin{equation}
    \partial_r \circ \Delta_0(H(r)\cdot W_i)=12 r^{-1} H(r)\cdot W_i\cdot 
    \begin{bmatrix}
0 & 5 & 0 & 0 & 0& 0\\
0 & 0 & 4 & 0 & 0& 0\\
0 & 0 & 0 & 0 & 0& 0\\
0 & 0 & 0 & 0 & 5& 0\\
0 & 0 & 0 & 0 & 0& 20\\
0 & 0 & 0 & 0 & 0& 0
\end{bmatrix}.
\end{equation}
We have
\begin{equation}
\begin{split}
    &\int_{\partial \bB_{\sigma}(0)\backslash \partial \bB_{\tau}(0)} \langle \Delta_0(H(r)\cdot W_i), \partial_r \circ \Delta_0(H(r)\cdot W_i) \rangle dA\\
    =&48 \Vol(\bS^3) \begin{bmatrix}
0 & 0 & 0 & 0 & 0& 0\\
0 & -15r^{-8} & -12r^{-6} & 0 & -15& -60r^2\\
0 & -20r^{-6} & -16r^{-4}& 0& -20r^2 & -80r^4\\
0 & 0 & 0 & 0 & 0& 0\\
0 & 25 & 20r^2 & 0 & 25r^8& 100r^{10}\\
0 & 60r^2 & 48r^4 & 0 & 60r^{10}& 240r^{12}
\end{bmatrix}\Big|_{r=\tau}^{r=\sigma}.
    \end{split}
\end{equation}
Combining them together yields
\begin{equation}
\begin{split}
    &\int_{\bB_{\sigma}(0)\backslash \bB_{\tau}(0)}\langle \nabla \Delta_0 (H(r)\cdot W_i), \nabla \Delta_0(H(r)\cdot W_i)\rangle \dvol\\
    =&48 \Vol(\bS^3) \begin{bmatrix}
0 & 0 & 0 & 0 & 0& 0\\
0 & -15r^{-8} & -20r^{-6} & 0 & 0& 60r^2\\
0 & -20r^{-6} & {-32r^{-4}}& 0& -20r^2 & 0\\
0 & 0 & 0 & 0 & 0& 0\\
0 & 0 & -20r^2 & 0 & 25r^8& 60r^{10}\\
0 & 60r^2 & 0 & 0 & 60r^{10}& 240r^{12}
\end{bmatrix}\Big|_{r=\tau}^{r=\sigma}.
\end{split}
\end{equation}
\end{proof}
\section{Energy estimate}\label{energy}
We continue the construction of the connected sum. Following the notation in Section \ref{connected-sum}. For formally defining the pasted submanifold, we remove from $\Sigma_1$ the neighborhood of the point $p_1$ corresponding to $\{z\in \bR^4:\gamma-\sqrt{\alpha}\leq |z|<\infty\}$ under the map $\pi\circ  \mathring{\Phi}_{1,\alpha}$. Analogously, we remove from $\Sigma_2$ the set corresponding to $\{z\in \bR^4: |z|\leq 1+\sqrt{\alpha}\}$ under the map $\pi\circ \Phi_{2, \beta^{-1}}$. We denote the remaining open submanifolds by $\mU\subset \Sigma_1$ and $\mV\subset \Sigma_2$, and put $\mW=\bB_{r\beta^{-1}}(0)\backslash \bB_{\alpha R}(0)$. Then the new submanifold is defined as $\Sigma:=\left(\mU\bigcup \mV\bigcup \mW\right)/ \sim$, where $\sim$ means the obvious identifications
\begin{equation*}
    \begin{split}
        p\sim z&=\pi\circ  \mathring{\Phi}_{1,\alpha}, \quad ~~\mbox{for}~~ p\in \mU, z\in \mW,\\
        q\sim z&=\pi\circ  \Phi_{2,\beta^{-1}}, \quad ~~\mbox{for}~~ q\in \mV, z\in \mW.
    \end{split}
\end{equation*}
$\Sigma$ is just the connected sum of $\Sigma_1 \# \Sigma_2$ and the new immersion is well defined piecewisely by setting
\begin{equation*}
\Phi: \Sigma\to \bR^n, \quad \Phi=\left\{
    \begin{split}
        &\mathring{\Phi}_{1,\alpha}, \quad ~~\mbox{on}~~ \mU,\\
        &\mathring{\Phi}_{2,\beta^{-1}}, \quad ~~\mbox{on}~~ \mV,\\
        &(z,w(z)),  \quad ~~\mbox{on}~~ \mW.
    \end{split}
    \right.
\end{equation*}
The rest of this section is devoted to estimate the energy of the immersion $\Sigma$.

\begin{lemma}[Triharmonic energy savings]\label{energy-savings}
    As $\alpha$ tends to zero, we have
    \begin{equation*}
        \begin{split}
            &\mB_{\infty, \gamma}(\mathring{p}_{\alpha},\mathring{p}_{\alpha})=\frac{\alpha^2}{\gamma^2}\Vol(\bS^3)\left(96\sum_{i=1}^{N_2}|p_i|^2+\frac{4\alpha^2}{3\gamma^2}\sum_{i=1}^{N_1}|q^{(1)}_i|^2 +\frac{128\alpha^2}{3\gamma^2} \sum_{i=1}^{N_3}|q^{(3)}_i|^2\right),\\
            &\mB_{\infty, \gamma-\sqrt{\alpha}}(\mathring{u}_{\alpha},\mathring{u}_{\alpha})=\mB_{\infty, \gamma}(\mathring{p}_{\alpha},\mathring{p}_{\alpha})+O\left(\frac{\alpha^{\frac{5}{2}}}{\gamma^3}\right)+O\left(\frac{\alpha^{4}}{\gamma^4}\right)+O\left(\frac{\alpha^{6}}{\gamma^6}\right);\\
            &\mB_{1, 0}(q_{\beta^{-1}},q_{\beta^{-1}})=4\Vol(\bS^3)\beta^4 \sum_{i=1}^{N_1} |s^{(1)}_i|^2,\\
            &\mB_{1+\sqrt{\alpha}, 0}(v_{\beta^{-1}},v_{\beta^{-1}})=\mB_{1, 0}(q_{\beta^{-1}},q_{\beta^{-1}})+O(\alpha^{\frac{1}{2}}\beta^2)+O(\beta^4)+O(\beta^6).
        \end{split}
    \end{equation*}
\end{lemma}
\begin{proof}
    Recall that 
    \begin{equation*}
        \mathring{u}_{\alpha}(z)=\mathring{p}_{\alpha}(z)+\mathring{\varphi}_{\alpha}(z),
    \end{equation*}
    where 
\begin{equation*}
    \mathring{p}_{\alpha}(z)=\frac{\alpha}{2}\mathring{P}_{\alpha}\left(\frac{z}{|z|},\frac{z}{|z|} \right)+\frac{\alpha^2}{6}{Q}_{\alpha}\left(\frac{z}{|z|},\frac{z}{|z|},\frac{z}{|z|} \right)|z|^{-1}.
\end{equation*}    
For $\gamma-\sqrt{\alpha}\leq |z|<\infty$, $\mathring{p}_{\alpha}(z)$ and $\mathring{\varphi}_{\alpha}(z)$ satisfy
\begin{equation}\label{apriori-estimate-u}
    \begin{split}
        \sum_{i=0}^{3} |z|^{i}|D^{i}\mathring{p}_{\alpha}(z)|\leq C\alpha,\quad \sum_{i=0}^{3} |z|^{i+2}|D^{i}\mathring{\varphi}_{\alpha}(z)|\leq C\alpha^3.
    \end{split}
\end{equation}
It follows that

\begin{equation}
    \begin{split}
        |\mB_{\infty, \gamma}(\mathring{p}_{\alpha},\mathring{\varphi}_{\alpha})|&\leq C\int_{\gamma}^{\infty}\frac{\alpha}{r^3}\cdot \frac{\alpha^3}{r^5}\cdot r^3dr=O\left(\frac{\alpha^4}{\gamma^4}\right),\\
        |\mB_{\infty, \gamma}(\mathring{\varphi}_{\alpha},\mathring{\varphi}_{\alpha})|&\leq C\int_{\gamma}^{\infty}\left(\frac{\alpha^3}{r^5}\right)^2 r^3dr=O\left(\frac{\alpha^6}{\gamma^6}\right),\\
        |\mB_{\gamma, \gamma-\sqrt{\alpha}}(\mathring{u}_{\alpha},\mathring{u}_{\alpha})|&=O\left(\frac{\alpha^{\frac{5}{2}}}{\gamma^3}\right).
    \end{split}
\end{equation}
Besides, by Lemma \ref{f}, \ref{g}, \ref{h} and (\ref{PQRS-expansion}), we have
\begin{equation*}
    \mathring{p}_{\alpha}(z)=\frac{\alpha}{2}\sum_{i=1}^{N_2}p_i\cdot G_3(r)\cdot Y_i+\frac{\alpha^2}{6} \left(\sum_{i=1}^{N_1} q^{(1)}_i\cdot F_2(r)\cdot X_i+\sum_{i=1}^{N_3}q^{(3)}_i \cdot H_3(r)\cdot W_i\right).
\end{equation*}
Therefore, by Lemma \ref{energy-g} and Lemma \ref{energy-h}, we obtain
\begin{equation}
    \mB_{\infty, \gamma}(\mathring{p}_{\alpha},\mathring{p}_{\alpha})=\Vol(\bS^3)\left(\frac{96\alpha^2}{\gamma^2}\sum_{i=1}^{N_2}|p_i|^2+\frac{4\alpha^4}{3\gamma^4}\sum_{i=1}^{N_1}|q^{(1)}_i|^2 +\frac{128\alpha^4}{3\gamma^4} \sum_{i=1}^{N_3}|q^{(3)}_i|^2\right).
\end{equation}
Similarly, we have
\begin{equation}
    v_{\beta^{-1}}(z)=q_{\beta^{-1}}(z)+\psi_{\beta^{-1}}(z) ,
\end{equation}
where
\begin{equation*}
    q_{\beta^{-1}}(z)=\frac{\beta}{2}R(z,z)+\frac{\beta^2}{6}S(z,z,z).
\end{equation*}
For $0\leq |z|\leq 1+\sqrt{\alpha}$, $q_{\beta^{-1}}(z)$ and $\psi_{\beta^{-1}}(z)$ satisfy
\begin{equation}\label{apriori-estimate-v}
\sum_{i=0}^{3} |D^{i}\mathring{q}_{\beta^{-1}}(z)|\leq C\beta,\quad \sum_{i=0}^{3} |z|^{-4+i}|D^{i}\psi_{\beta^{-1}}(z)|\leq C\beta^3 
\end{equation}
It follows that
\begin{equation}
    \begin{split}
        |\mB_{1, 0}(q_{\beta^{-1}},\psi_{\beta^{-1}})|&\leq C\beta^4\int_{0}^{1} r^4dr=O(\beta^4),\\
        |\mB_{1, 0}(\psi_{\beta^{-1}},\psi_{\beta^{-1}})|&\leq C\beta^6\int_{0}^{1} r^5dr=O(\beta^6),\\
        |\mB_{1+\sqrt{\alpha}, 1}(v_{\beta^{-1}},v_{\beta^{-1}})|&=O(\alpha^{\frac{1}{2}}\beta^2).
    \end{split}
\end{equation}
Besides, by Lemma \ref{e}, \ref{f}, \ref{g}, \ref{h} and (\ref{PQRS-expansion}), we have
\begin{equation*}
\begin{split}
    q_{\beta^{-1}}(z)=&\frac{\beta}{2}\left(r_0e_4(r)+\sum_{i=1}^{N_2}r_i\cdot G_4(r)\cdot Y_i\right)\\
    &+\frac{\beta^2}{6} \left(\sum_{i=1}^{N_1} s^{(1)}_i\cdot F_5(r)\cdot X_i+\sum_{i=1}^{N_3}s^{(3)}_i\cdot H_4(r)\cdot W_i\right).
    \end{split}
\end{equation*}
Therefore, by Lemma \ref{energy-e}, \ref{energy-f}, \ref{energy-g} and \ref{energy-h}, we obtain
\begin{equation}
\begin{split}
    \mB_{1, 0}(q_{\beta^{-1}},q_{\beta^{-1}})=4\Vol(\bS^3)\beta^4 \sum_{i=1}^{N_1} |s^{(1)}_i|^2.
    \end{split}
\end{equation}
\end{proof}

\begin{lemma}\label{triharmonic-energy}
The triharmonic energy of the interpolation is given by as follows
    \begin{equation*}
        \begin{split}
            \mB_{1, \gamma}(w,w)=&32\Vol(\bS^3)\beta^2|r_0|^2\gamma^2\\            
            &+3\Vol(\bS^3)\frac{\alpha^4}{\gamma^4}\sum_{i=1}^{N_1}|q^{(1)}_i|^2\left(\frac{2}{3}-\frac{3}{2(\ln \gamma)^2}\right)^2\\
    &+4\Vol(\bS^3)\beta^4\sum_{i=1}^{N_1}|s^{(1)}_i|^2+\frac{8}{3}\Vol(\bS^3)\alpha^2\beta^2 \sum_{i=1}^{N_1}\langle q^{(1)}_i, s^{(1)}_i\rangle \\
&+96\Vol(\bS^3) \alpha^2\sum_{i=1}^{N_2}\left(\gamma^{-2}+3\right)|p_i|^2\\
   &-192\Vol(\bS^3)\alpha\beta\sum_{i=1}^{N_2} \langle p_i, r_i\rangle\\
    &+\frac{128}{3}\Vol(\bS^3)\frac{\alpha^4}{\gamma^4}\sum_{i=1}^{N_3}|q^{(3)}_i|^2+\frac{2368}{3}\Vol(\bS^3)\alpha^4\sum_{i=1}^{N_3}|q^{(3)}_i|^2\\
&-\frac{256}{3}\Vol(\bS^3)\alpha^2\beta^2\sum_{i=1}^{N_3}\langle q^{(3)}_i, s^{(3)}_i\rangle\\
&+O(\beta^2\gamma^4(\ln \gamma)^2)+O(\alpha^4(\ln \gamma)^{-1})+O(\alpha^2 \gamma^2)+O(\alpha^4\gamma^2),\\
            |\mB_{\gamma, \gamma-\sqrt{\alpha}}(w,w)|&+|\mB_{1+\sqrt{\alpha}, 1}(w,w)|=O\left(\frac{\alpha^{\frac{5}{2}}}{\gamma^3}\right).
        \end{split}
    \end{equation*}
\end{lemma}
\begin{proof}
    Recall that on $B_1(0)\backslash B_{\gamma}(0)$
        \begin{equation*}
w=e(r)+\sum_{i=1}^{N_1}f_i(r)X_i+\sum_{i=1}^{N_2}g_i(r)Y_i+\sum_{i=1}^{N_3}h_i(r)W_i.
\end{equation*}
Therefore, by Lemma \ref{energy-e}, \ref{energy-f}, \ref{energy-g} and \ref{energy-h}, we have
\begin{equation*}
    \begin{split}
        \mB_{1, \gamma}(w,w)&=\mB_{1, \gamma}(e(r),e(r))+\sum_{i=1}^{N_1}\mB_{1, \gamma}(f_i(r)X_i,f_i(r)X_i)\\
        &+\sum_{i=1}^{N_2}\mB_{1, \gamma}(g_i(r)Y_i,g_i(r)Y_i)+\sum_{i=1}^{N_3}\mB_{1, \gamma}(h_i(r)W_i,h_i(r)W_i).
    \end{split}
\end{equation*}
By Lemma \ref{corfficient-e} and Lemma \ref{energy-e}, we have
\begin{equation*}
    \mB_{1, \gamma}(e(r),e(r))=32\Vol(\bS^3)\beta^2|r_0|^2\gamma^2+O(\beta^2\gamma^4(\ln \gamma)^2).
\end{equation*}

By Lemma \ref{corfficient-f} and Lemma \ref{energy-f}, we have
\begin{equation*}
\begin{split}
    \mB_{1, \gamma}(f_i(r)X_i,f_i(r)X_i)=&3\Vol(\bS^3)\frac{\alpha^4}{\gamma^4}|q^{(1)}_i|^2\left(\frac{2}{3}-\frac{3}{2(\ln \gamma)^2}\right)^2\\
    &+4\Vol(\bS^3)\beta^4|s^{(1)}_i|^2+\frac{8}{3}\Vol(\bS^3)\alpha^2\beta^2 \langle q^{(1)}_i, s^{(1)}_i\rangle \\
    &+O(\alpha^4(\ln \gamma)^{-1}).
    \end{split}
\end{equation*}

By Lemma \ref{corfficient-g} and Lemma \ref{energy-g}, we have
\begin{equation*}
\begin{split}
   \mB_{1, \gamma}(g_i(r)Y_i,g_i(r)Y_i)=&96\Vol(\bS^3) \alpha^2\left(\gamma^{-2}+3\right)|p_i|^2\\
   &-192\Vol(\bS^3)\alpha\beta \langle p_i, r_i\rangle+O(\alpha^2 \gamma^2).
   \end{split}
\end{equation*}

By Lemma \ref{corfficient-h} and Lemma \ref{energy-h}, we have

\begin{equation*}
         %128\Vol(\bS^3)\alpha^2 p_i^2\times\\
         %&\left(-\frac{3}{4}(1-\gamma^{-2})-3(1-\gamma^2)+9(1-\gamma^6)-15(1-\gamma^8)+7.5(1-\gamma^{10})\right)+h.o.t\\    
\begin{split}
\mB_{1, \gamma}(h_i(r)W_i,h_i(r)W_i)=&\frac{128}{3}\Vol(\bS^3)\frac{\alpha^4}{\gamma^4}|q^{(3)}_i|^2+\frac{2368}{3}\Vol(\bS^3)\alpha^4|q^{(3)}_i|^2\\
&-\frac{256}{3}\alpha^2\beta^2\Vol(\bS^3)\langle q^{(3)}_i, s^{(3)}_i\rangle+O(\alpha^4\gamma^2).
\end{split}
\end{equation*}

For $\gamma-\sqrt{\alpha}\leq |z|\leq 1+\sqrt{\alpha}$, $w$ satisfies
\begin{equation}\label{apriori-estimate-w}
    \sum_{i=0}^{3}|D^i w|\leq O(\alpha, \beta).
\end{equation}
Therefore, we obtain that
\begin{equation*}
    |\mB_{\gamma, \gamma-\sqrt{\alpha}}(w,w)|+|\mB_{1+\sqrt{\alpha}, 1}(w,w)|=O\left(\frac{\alpha^{\frac{5}{2}}}{\gamma^3}\right).
\end{equation*}
\end{proof}

For clarity, the fourth order Willmore energy considered is 
\begin{equation*}
   \mE^{(\mu,\nu)}:=\mE_{GR}+\mu \int_{\Sigma} |\mathring{L}|^4\dvol+\nu \int_{\Sigma}|\mathring{L}^2|^2 \dvol, \quad \mu, \nu \in \bR.
\end{equation*}
In slight abuse of notation, we denote by $ \mE^{(\mu,\nu)}_{\sigma,\tau}(\Phi)$ the energy $ \mE^{(\mu,\nu)}$ of the graph of the immersion $\Phi$ restricted to $B_{\sigma}(0)\backslash B_{\tau}(0)$.
\begin{lemma}
Let $\alpha=\gamma^8$ and $\beta=t\alpha$. As $\gamma$ tends to $0$, we have the expansion
    \begin{equation}\label{difference-of-energy}
    \begin{split}
    \mE^{(\mu,\nu)}(\Phi)&-\left(\mE^{(\mu,\nu)}(\Phi_1)+\mE^{(\mu,\nu)}(\Phi_2)-8\pi^2\right)\\
    =&\frac{1}{16}\left(\mB_{1, \gamma}(w,w)-\mB_{\infty, \gamma-\sqrt{\alpha}}(\mathring{u}_{\alpha},\mathring{u}_{\alpha})-\mB_{1, 0}(q_{\beta^{-1}},q_{\beta^{-1}})\right)+O(\gamma^{17})\\
    =&18\Vol(\bS^3) \gamma^{16}\sum_{i=1}^{N_2}|p_i|^2
            -12t\Vol(\bS^3)\gamma^{16}\sum_{i=1}^{N_2}\langle p_i, r_i\rangle + O(\gamma^{17}).
    \end{split}
\end{equation}
\end{lemma}
\begin{proof}
By the construction in Section \ref{connected-sum}, we know that
\begin{equation}
\begin{split}
    \mE^{(\mu,\nu)}(\Phi)&-\left(\mE^{(\mu,\nu)}(\Phi_1)+\mE^{(\mu,\nu)}(\Phi_2)-8\pi^2\right)\\
    &=\mE^{(\mu,\nu)}_{1+\sqrt{\alpha},\gamma-\sqrt{\alpha}}(w)-\mE^{(\mu,\nu)}_{\infty,\gamma-\sqrt{\alpha}}(\mathring{u}_{\alpha})-\mE^{(\mu,\nu)}_{1+\sqrt{\alpha},0}(v_{\beta^{-1}})
    \end{split}
\end{equation}

Using apriori estimates (\ref{apriori-estimate-u}), (\ref{apriori-estimate-v}) and (\ref{apriori-estimate-w}), we have
    \begin{equation}
    \int_{|z|\geq \gamma-\sqrt{\alpha}} |L|^4\dvol\leq C\int_{\gamma-\sqrt{\alpha}}^{\infty}\left(\frac{\alpha}{r^2}\right)^4 r^3dr=O(\gamma^{28}),
\end{equation}
\begin{equation}
    \int_{0\leq|z|\leq 1+\sqrt{\alpha}} |L|^4\dvol\leq C\beta^4\int_{0}^{1} r^3dr= O(\gamma^{32}),
\end{equation}
and 
\begin{equation}
    \int_{\gamma-\sqrt{\alpha}\leq|z|\leq 1+\sqrt{\alpha}}|L|^4\dvol\leq O(\gamma^{32}).
\end{equation}
Besides, apriori estimates (\ref{apriori-estimate-w}) verifies the assumptions of Lemma \ref{technique--lemma}. Combining with Lemma \ref{energy-savings} and \ref{triharmonic-energy} yields
\begin{equation*}
    \begin{split}
    \mE^{(\mu,\nu)}(\Phi)&-\left(\mE^{(\mu,\nu)}(\Phi_1)+\mE^{(\mu,\nu)}(\Phi_2)-8\pi^2\right)\\
    =&\mE^{(\mu,\nu)}_{1+\sqrt{\alpha},\gamma-\sqrt{\alpha}}(w)-\mE^{(\mu,\nu)}_{\infty,\gamma-\sqrt{\alpha}}(\mathring{u}_{\alpha})-\mE^{(\mu,\nu)}_{1+\sqrt{\alpha},0}(v_{\beta^{-1}})\\
    =&\frac{1}{16}\left(\mB_{1, \gamma}(w,w)-\mB_{\infty, \gamma-\sqrt{\alpha}}(\mathring{u}_{\alpha},\mathring{u}_{\alpha})-\mB_{1, 0}(q_{\beta^{-1}},q_{\beta^{-1}})\right)+O(\gamma^{17})\\
    =&18\Vol(\bS^3) \gamma^{16}\sum_{i=1}^{N_2}|p_i|^2
            -12t\Vol(\bS^3)\gamma^{16}\sum_{i=1}^{N_2}\langle p_i, r_i\rangle + O_t(\gamma^{17}).
    \end{split}
\end{equation*}

\end{proof}
\section{Proof of Main Results}\label{Conclusion}
In this section, we will finish the proof of Theorem \ref{Main 1}. It should be noted that if the origin has more than one preimage under $\Phi_1$, we can carry out the whole construction on one end of $\mathring{\Phi}_1$ and compactify other ends via inversions as Bauer and Kuwert did in \cite{Bauer-Kuwert03}. Therefore, we know that Theorem \ref{Main 1} will follow provided that 
\begin{equation}\label{Matrix multi}
   \sum_{i=1}^{N_2}\langle p_i, r_i\rangle= \langle \mathring{P}, \mathring{R} \rangle > 0.
\end{equation}
The following crucial lemma allows us to choose a suitable rotation for the constructed connected sum so that \ref{Matrix multi} holds.
\begin{lemma}
    Let $\mathring{P}, \mathring{R}: \mathbb{R}^4 \times \mathbb{R}^4 \to \mathbb{R}^k $ be bilinear forms which are symmetric tracefree. Then, there exists orthogonal transformations $S \in \mathbb{O}(4) $ and $T \in \mathbb{O}(k)$ such that the form $\mathring{P}_{S, T}(\zeta, \zeta)=T\mathring{P}(S^{-1} \zeta, S^{-1} \zeta)$ satisfies $\langle \mathring{P}_{S, T}, \mathring{R} \rangle  \geq 0.$ Furthermore, if $\mathring{P}, \mathring{R}: \mathbb{R}^4 \times \mathbb{R}^4 \to \mathbb{R}^k $ both nonzero, then  $\langle \mathring{P}_{S, T}, \mathring{R} \rangle  > 0.$
\end{lemma}
\begin{proof}
Since $\mathring{P}$ takes values in $\mathbb{R}^k$,  we can  express $\mathring{P}$ as $\mathring{P}=\left(\mathring{P}_1, \mathring{P}_2, \ldots, \mathring{P}_k\right)$, where each $\mathring{P}_i: \mathbb{R}^4 \times$ $\mathbb{R}^4 \rightarrow \mathbb{R}$ is a scalar-valued symmetric bilinear form. $\mathring{P}_i$ corresponds to a symmetric $4 \times 4$ matrix (also denoted $\mathring{P}_i$ ) via $\mathring{P}_i(u, v)=u^{\top} \mathring{P}_i v$. Moreover, $\operatorname{Tr}\left(\mathring{P}_i\right)=0$ ensures that no component has a bias toward the identity matrix. Similarly, we have
$\mathring{R}=\left(\mathring{R}_1, \mathring{R}_2, \ldots, \mathring{R}_k\right)$, with each $\mathring{R}_i$ symmetric and tracefree.

For $S \in \mathbb{O}(4)$, the transformation $S^{-1} \zeta$ changes the basis of $\mathbb{R}^4$.
The transformed bilinear form becomes:
\begin{equation}
\mathring{P}^S(\zeta, \zeta):=\mathring{P}\left(S^{-1} \zeta, S^{-1} \zeta\right)=\left(S \mathring{P}_1 S^{\top}, \ldots, S \mathring{P}_k S^{\top}\right), 
\end{equation}
where $S$ acts on each component $\mathring{P}_i$.
For $T \in \mathbb{O}(k)$, the transformation is given by:
\begin{equation}
\mathring{P}_{S, T}(\zeta, \zeta):=T \mathring{P}^S(\zeta, \zeta)=\left(\sum_{j=1}^k T_{1 j} S \mathring{P}_j S^{\top}, \ldots, \sum_{j=1}^k T_{k j} S \mathring{P}_j S^{\top}\right). 
\end{equation}
Then the inner product between $\mathring{P}_{S, T}$ and $\mathring{R}$ is:
\begin{equation}
\left\langle \mathring{P}_{S, T}, \mathring{R}\right\rangle=\sum_{i=1}^k \operatorname{Tr}\left(\left(\sum_{j=1}^k T_{i j} S \mathring{P}_j S^{\top}\right) \mathring{R}_i\right).
\end{equation}
By simplification, it becomes:
\begin{equation}
\left\langle \mathring{P}_{S, T}, \mathring{R}\right\rangle=\sum_{i, j=1}^k T_{i j} \operatorname{Tr}\left(S \mathring{P}_j S^{\top} \mathring{R}_i\right).
\end{equation}
For fixed $S \in \mathbb{O}(4)$, we define the matrix $A(S) \in \mathbb{R}^{k \times k}$ with entries $A_{i j}(S)=\operatorname{Tr}\left(S \mathring{P}_j S^{\top} \mathring{R}_i\right)$. The inner product $\left\langle \mathring{P}_{S, T}, \mathring{R}\right\rangle$ becomes $\operatorname{Tr}\left(T^{\top} A(S)\right)$.

\textbf{Claim:} the maximum of $\operatorname{Tr}\left(T^{\top} A(S)\right)$ over $ \mathbb{O}(k)$ can be achieved by some matrix $T$ and is given by the trace norm of $A(S)$: 
\begin{equation*}
\max _{T \in \mathbb{O}(k)} \operatorname{Tr}\left(T^{\top} A(S)\right)=\|A(S)\|_{\mathrm{tr}}:=\sum_{\alpha'=1}^{k} \sigma_{\alpha'}(S),    
\end{equation*}
where $\sigma_{\alpha'}(S)$ are the singular values of $A(S)$.

To prove this, we first apply the Singular Value Decomposition (SVD) for the matrix $A(S)$:
we can rewrite $A(S)$ as $U \Sigma V^{\top}$ (the SVD of $A(S)$), where
\begin{enumerate}
    \item $U, V \in \mathbb{O}(k)$ are orthogonal matrices;
    \item  $\Sigma$ is a diagonal matrix with non-negative non-increasing entries $\sigma_1 \geqslant  \sigma_2 \geqslant \cdots \geqslant \sigma_k \geqslant 0$ (the singular values).
\end{enumerate}
Substituting $A(S)$ by $U \Sigma V^{\top}$ in $\operatorname{Tr}\left(T^{\top} A(S)\right)$, we obtain
\begin{equation}
\operatorname{Tr}\left(T^{\top} A(S)\right)=\operatorname{Tr}\left(T^{\top} U \Sigma V^{\top}\right). 
\end{equation}
Using the cyclic property of the trace $(\operatorname{Tr}(A B C)=\operatorname{Tr}(B C A))$, we have 
\begin{equation}
\operatorname{Tr}\left(T^{\top} A(S)\right)=\operatorname{Tr}(\Sigma W), \quad W:=V^{\top} T^{\top} U.
\end{equation}
The trace $\operatorname{Tr}(W \Sigma)$ is the sum of the diagonal entries of $W \Sigma$ :
\begin{equation}
\operatorname{Tr}( \Sigma W)=\sum_{i=1}^k W_{\alpha' \alpha'} \sigma_{\alpha'}  
\end{equation}
By the Cauchy--Schwarz inequality and the orthogonality of $W$, it is easy to check that each diagonal entry $W_{\alpha' \alpha'}$ satisfies $\left|W_{\alpha' \alpha'}\right| \leqslant 1$. So the maximum of $\operatorname{Tr}\left(T^{\top} A(S)\right)$ is achieved when $W$ is the identity matrix:
\begin{equation}
\max _{W \in \mathbb{O}(k)} \operatorname{Tr}( \Sigma W)=\sum_{\alpha'=1}^k \sigma_{\alpha'} . 
\end{equation}
Equivalently, the maximum of $\operatorname{Tr}\left(T^{\top} A(S)\right)$ over $\mathbb{O}(k)$ is:
\begin{equation}
\max _{T \in \mathbb{O}(k)} \operatorname{Tr}\left(T^{\top} A(S)\right)=\sum_{\alpha'=1}^k \sigma_{\alpha'}=\|A(S)\|_{\mathrm{tr}}. 
\end{equation}
It can be realized by some matrix $T \in \mathbb{O}(k)$ because of the compactness of $\mathbb{O}(k)$.
The claim follows immediately.

Since all $\sigma_{\alpha'}$ is non-negative, the trace norm $\|A(S)\|_{\mathrm{tr}}$ is always non-negative. This proves the first part of the lemma.

Furthermore, the norm $\|\cdot\|_{\mathrm{tr}}$ vanishes if and only if all singular values are zeros, which implies $A(S)=0$ for any $S \in \mathbb{O}(4)$. Recall that the orthogonal group $\mathbb{O}(4)$ acts on the space of traceless symmetric $4\times4$ matrices via conjugation:
$P_j \mapsto S P_j S^{\top}$, $S \in \mathbb{O}(4)$ and this is an irreducible representation of $\mathbb{O}(4)$: the set $\left\{S P_j S^{\top} \mid S \in O(4)\right\}$ spans the entire space of traceless symmetric
$4\times4$ matrices for any non zero $P_j$, $1\leq j\leq k$.

Now suppose $\mathring{P}, \mathring{R}: \mathbb{R}^4 \times \mathbb{R}^4 \to \mathbb{R}^k $ are both non-vanishing matrices. If $A(S)=0$ for any $S \in \mathbb{O}(4)$, there exists $\mathring{P}_{j_0} \neq 0$ and $\mathring{R}_{i_0} \neq 0$ satisfying $\operatorname{Tr}\left(S \mathring{P}_{j_0} S^{\top} \mathring{R}_{i_0}\right)=0$ for any $S \in O(4)$. Since $\left\{S \mathring{P}_{j_0} S^{\top} \mid S \in O(4)\right\}$ spans the entire space of traceless symmetric
$4\times4$ matrices, we get that 
$\operatorname{Tr}( P \mathring{R}_{i_0})=0$ for all traceless symmetric
$4\times4$ matrices $P$, which implies that $\mathring{R}_{i_0}=0$, a contradiction!
Hence, $\|A(S)\|_{\mathrm{tr}}>0$ for some $S \in \mathbb{O}(4)$. The proves the second part of the lemma.
 \end{proof}

\subsection*{Acknowledgements}

The authors would like to thank Yuchen Bi, Tzu-Mo Kuo, Xuezhang Chen, Qing Han, Xianzhe Dai, Ao Sun, Guofang Wei and Jie Zhou for useful discussions. Thank Jeffrey Case for fruitful comments and pointing out relevant references. We also extend our gratitude to Dorian Martino for identifying typos and providing valuable corrections to this work. ZY is supported by a Simons Travel Grant.

\bibliography{mybibnew}
\bibliographystyle{alpha}

\end{document}